\documentclass[a4paper, 12pt]{amsart}

\usepackage{amssymb}
\usepackage{amsthm}
\usepackage{amsmath}

\usepackage[mathscr]{eucal}

\usepackage{enumitem}

\usepackage{stmaryrd}
\usepackage{manfnt}

\theoremstyle{plain}
\newtheorem{thm}{Theorem}[section]
\newtheorem{prop}[thm]{Proposition}
\newtheorem{cor}[thm]{Corollary}
\newtheorem{lem}[thm]{Lemma}

\theoremstyle{definition}
\newtheorem{df}{Definition}[section]
\newtheorem{exam}{Example}[section]
\newtheorem{ques}[thm]{Question}

\theoremstyle{remark}
\newtheorem{rmk}{Remark}[section]
\newtheorem*{ac}{Acknowledgements}

\newcommand{\To}{\Longrightarrow}

\newcommand{\zz}{\mathbb{Z}}

\newcommand{\rr}{\mathbb{R}}

\newcommand{\mm}{\zz_{\ge 0}}

\newcommand{\yoemph}[1]{\emph{#1}}

\newcommand{\yoomega}{\omega_{0}}

\newcommand{\yourysp}{\mathbb{U}}
\newcommand{\yourydis}{\rho}

\DeclareMathOperator{\card}{Card}

\newcommand{\yosub}{\subseteq}

\DeclareMathOperator{\yodiam}{diam}

\newcommand{\yocball}{\mathrm{B}}

\newcommand{\yochara}{\mathscr}

\newcommand{\youfin}{\yochara{N}}

\newcommand{\yofin}{\yochara{F}}

\newcommand{\yoofam}[1]{\mathbf{TEN}(#1)}

\newcommand{\yopetal}[2]{\Pi(#1, #2)}
\newcommand{\yopetalq}[1]{\Pi(#1)}

\newcommand{\yomapsco}[2]{\mathrm{G}(#1, #2)}
\newcommand{\yomaindisco}{\mathord{\vartriangle}}

\newcommand{\yoexpsp}[1]{\mathscr{K}(#1)}

\newcommand{\yoexpspml}[3]{\mathscr{K}_{#2, #3}(#1)}

\newcommand{\yoexpdis}[1]{\mathcal{HD}_{#1}}

\newcommand{\yobrasp}[2]{\langle #1 \rangle_{\mathfrak{c}(#2)}}

\newcommand{\yoseed}[2]{\Omega[#1, #2]}

\newcommand{\yoxsp}[1]{\mathbb{A}(#1)}
\newcommand{\yopsp}[1]{\mathbb{E}(#1)}

\newcommand{\yoprotobio}{\varpi}

\newcommand{\yorel}{\mathbin{\sim_{r}}}

\newcommand{\yocase}[2]{Case #1.~[#2]:}

\makeatletter
\@addtoreset{equation}{section}

\makeatother

\begin{document}

\title[Urysohn universal ultrametric spaces]
{
Characterizations of 
  Urysohn universal 
ultrametric spaces
}

\author[Yoshito Ishiki]
{Yoshito Ishiki}

\address[Yoshito Ishiki]
{\endgraf
Photonics Control Technology Team
\endgraf
RIKEN Center for Advanced Photonics
\endgraf
2-1 Hirasawa, Wako, Saitama 351-0198, Japan}

\email{yoshito.ishiki@riken.jp}

\date{\today}

\begin{abstract}
In this paper, 
using the existence of infinite  equidistant subsets of  closed balls, 
we first characterize the injectivity of ultrametric spaces for  finite ultrametric spaces. 
This method 
also gives  characterizations of 
the 
 Urysohn universal ultrametric spaces. 
As an application,
we find that 
the operations of 
the Cartesian  product and 
the hyperspaces preserve 
the structures of 
the Urysohn universal ultrametric spaces. 
Namely, 
let 
$(X, d)$ 
be the 
Urysohn universal 
ultrametric space. 
Then 
 we show 
that 
$(X\times X, d\times d)$ 
is 
isometric to 
$(X, d)$, 
and 
show  that
 the hyperspace
  consisting of all non-empty compact subsets 
  of  
$(X, d)$ 
and symmetric products of $(X, d)$
are isometric to  
$(X, d)$. 
We also establish that every 
complete ultrametric space injective for 
finite ultrametric space contains 
a subspace isometric to $(X, d)$. 
\end{abstract}


\maketitle

\section{Introduction}\label{sec:intro}
\subsection{Backgrounds}
For a class 
$\yochara{C}$ 
of metric spaces, 
we say that a metric space $(X, d)$ is 
\yoemph{$\yochara{C}$-injective}
or 
\yoemph{injective for $\yochara{C}$}
if it is non-empty and  for every  pair 
of metric spaces 
$(Y, e)$ 
and 
$(Z, h)$ 
in 
$\yochara{C}$,  
and  for every pair of
isometric embeddings 
$\phi\colon Y\to Z$ 
and 
$f\colon Y\to X$, 
there exists an isometric embedding 
$F\colon Z\to X$ 
such that 
$F\circ \phi=f$. 
In other words, 
$(X, d)$ is $\yochara{C}$-injective if and only if 
for every metric 
space
$(Z, h)$ in 
$\yochara{C}$, and 
for 
every metric subspace 
$(Y, e)$ of 
$(Z, h)$ 
with 
$(Y, e)\in \yochara{C}$, 
every isometric embedding from 
$(Y, e)$ into $(X, d)$ can be 
extended to 
an isometric embedding 
from $(Z, h)$ into 
$(X, d)$. 
We denote by $\yofin$ the class of all finite metric spaces. 
Urysohn 
\cite{Ury1927} 
constructed a separable complete  $\yofin$-injective metric space
and proved its uniqueness up to isometry. 
This  space is nowadays called the 
\yoemph{Urysohn universal (metric) space}, 
and it has been studied in 
geometry, topology, dynamics, and model theory. 
In what follows, 
let 
$(\yourysp, \yourydis)$ denote  the Urysohn universal metric space. 
Our main subjects of this paper are
non-Archimedean analogues of 
$(\yourysp, \yourydis)$. 
For more information of 
$(\yourysp, \yourydis)$ 
and related topics, 
we refer the readers to, for instance,  
\cite{MR1843595}
\cite{MR2667917},
\cite{MR2435145}, 
\cite{MR952617}, 
\cite{MR2435148}, 
\cite{MR2051102}, 
\cite{MR1900705},
\cite{MR2277969}, 
\cite{MR3583613}, 
and 
\cite{MR4282005}.

To explain our main results and backgrounds, we prepare some notations and notions. 
A metric 
$d$ on  a set 
$X$ is 
said to be an 
\yoemph{ultrametric} if 
it satisfies the so-called strong triangle inequality 
$d(x, y)\le d(x, z)\lor d(z, y)$ 
for all 
$x, y, z\in X$, 
where 
$\lor$ 
stands for the 
maximal operator on 
$\rr$. 
We say that a set 
$R$  
is 
a
 \yoemph{range set} 
 if it is a subset of 
 $[0, \infty)$
and 
$0\in R$.
For a range set
 $R$, 
an ultrametric space 
$(X, d)$
 is 
\yoemph{$R$-valued} 
if 
$d(x, y)\in R$ 
for all 
$x, y\in R$. 
For a range set 
$R$, 
we  denote by 
$\youfin(R)$ 
the class of 
all finite 
$R$-valued ultrametric spaces. 

For a finite or   countable   range set 
$R$, 
we say that an ultrametric 
$(X, d)$
 is 
the 
\yoemph{$R$-Urysohn universal ultrametric space} 
if 
it is a separable complete 
$\youfin(R)$-injective
 $R$-valued 
ultrametric space.
Similarly to  
$(\yourysp, \yourydis)$, 
in the case where $R$ is finite or  countable, 
using the so-called back-and-forth argument, 
we can prove the uniqueness of 
the 
$R$-Urysohn universal ultrametric space up to 
isometry
(see \cite{MR2754373} and \cite{MR1843595}). 

In \cite{MR2754373}, 
Gao and Shao provided 
various
constructions of the 
$R$-Urysohn universal ultrametric spaces
for every finite or countable range set 
$R$. 
Wan \cite{MR4282005} 
established that 
the space of all non-empty compact 
$R$-valued ultrametric spaces 
equipped with the 
Gromov--Hausdorff ultrametric
is the 
$R$-Urysohn universal ultrametric space
for every finite or  countable range set 
$R$. 
In 
\cite{Ishiki2023const}, 
the author gave other constructions of 
Urysohn universal ultrametric spaces using 
mapping spaces and  spaces of 
continuous pseudo-ultrametrics.

Since for every uncountable range set 
$R$
all  
$\youfin(R)$-injective ultrametric spaces are 
non-separable 
(this is a simple corollary of 
\cite[(12) in Theorem 1.6]{MR3782290}), 
mathematicians often treat
the 
$R$-Urysohn universal ultrametric space only in  the case 
where 
$R$ 
is finite or  countable 
to guarantee the separability. 
In 
\cite{Ishiki2023Ury}, 
to expand the theory of 
Urysohn universal space into the non-separable spaces, 
for every uncountable range set 
$R$, 
 the author introduced 
the 
\yoemph{$R$-petaloid ultrametric space}, 
which 
is intended to be a standard
class of 
non-separable Urysohn universal ultrametric spaces. 
The author also proved its uniqueness up to isometry
using the back-and-forth argument together with transfinite induction
(see \cite[Theorem 1.1]{Ishiki2023Ury}  
and Theorem \ref{thm:18:petapeta}
in the present paper).
The definition and basic properties  of the petaloid spaces
will be presented  in  Section 
\ref{sec:pre}
 in this paper.

Based on  the notion of the petaloid 
ultrametric space, 
even when 
$R$ 
is uncountable, 
we introduce 
the 
$R$-Urysohn universal ultrametric space 
as the  
$R$-petaloid ultrametric space. 
Namely, 
in this paper, 
 we define the concept of 
 the 
 $R$-Urysohn universal ultrametric space as follows. 
\begin{df}\label{df:18:origin}
Let 
$R$ 
be a range set. 
When 
$R$ is finite or  countable, 
the 
\yoemph{$R$-Urysohn universal ultrametric space}  means 
 a (unique) 
 separable complete 
 $\youfin(R)$-injective 
 $R$-valued ultrametric space. 
When 
$R$ 
is uncountable, 
the 
\yoemph{$R$-Urysohn universal ultrametric space} 
 is the 
 $R$-petaloid ultrametric space. 
\end{df}

As a development of 
the theory of the 
Urysohn universal ultrametric spaces 
(and the 
 petaloid spaces), 
in the present  paper, 
we investigate 
characterizations of 
injective spaces for finite ultrametric spaces,  
and 
we prove that 
the Cartesian product 
and the hyperspaces preserve the 
structures of the Urysohn universal ultrametric spaces. 
We also establish that 
every complete  
$\youfin(R)$-injective 
ultrametric spaces contains 
a subspace isometric to 
the 
$R$-Urysohn universal ultrametric space.

\subsection{Main results}
We now explain our main results. 
We prepare some  definitions.

We denote by 
$\yoomega$ 
the set of all 
non-negative integers according to 
a
 set-theoretical notation. 
Remark that 
$\yoomega=\zz_{\ge 0}$ 
as a set. 
We sometimes  write 
$n<\yoomega$ 
if 
$n\in \yoomega$. 
Thus, 
the relation 
$\kappa\le \yoomega$ 
means that 
$\kappa\in \yoomega$ 
or 
$\kappa=\yoomega$. 

For a set 
$E$, 
we write
$\card(E)$ 
as  the 
cardinality of
 $E$.

\begin{df}\label{df:18:N}
For a cardinal 
$\kappa$,
and for a range set 
$R$, 
 we denote by 
$\youfin(R, \kappa)$ 
the class of all 
$R$-valued 
ultrametric spaces 
$(X, d)$ 
such that 
$\card(X)<\kappa$. 
Remark that if  
$\kappa=\yoomega$, 
then 
$\youfin(R, \kappa)$ 
coincides with the 
class $\youfin(R)$. 
In what follows, 
we use 
$\youfin(R, \yoomega)$ 
rather than 
$\youfin(R)$. 
\end{df}

For a metric space 
$(X, d)$, 
for 
$a\in X$, 
and  for 
$r\in (0, \infty)$, 
we denote by 
$\yocball(a, r; d)$ 
the closed ball 
centered at 
$a$ 
with  radius 
$r$.
We often simply write 
$\yocball(a, r)$
instead of 
$\yocball(a, r; d)$.

For a metric space 
$(X, d)$ 
and for 
$r\in (0, \infty)$, 
a subset 
$A$ 
of 
$X$ 
is called an 
\yoemph{$r$-equidistant set} 
if 
$d(x, y)=r$ 
for all distinct 
$x, y\in A$.

\begin{df}\label{df:18:halo}
Let 
$\kappa\le \yoomega$ 
and 
$R$ 
be a range set. 
An ultrametric space 
$(X, d)$ 
is 
said to be 
\yoemph{$(R, \kappa)$-haloed} 
if 
for every  
$a\in X$ 
and for every 
$r\in R\setminus \{0\}$, 
there exists  an 
$r$-equidistant  subset 
$A$ 
of 
$\yocball(a, r)$ 
such that 
$\kappa\le \card(A)$. 
\end{df}

\begin{df}\label{df:18:avoidant}
Let 
$\kappa\le \yoomega$ 
and 
$R$ 
be a range set. 
An  ultrametric space 
$(X, d)$ 
is 
\yoemph{$(R, \kappa)$-avoidant} 
if 
for every 
$a\in X$, 
 every 
 $r\in R\setminus \{0\}$, 
 and for every  subset  
 $A$ 
 of 
 $\yocball(a, r)$
  with 
  $\card(A)<\kappa$, 
 there exists 
 $p\in \yocball(a, r)$ 
 such that 
 $d(x, p)=r$ 
 for all
  $x\in A$.
\end{df}
The following theorem is our first main result. 
\begin{thm}\label{thm:18:chara}
Let 
$R$ 
be a range set. 
Let 
$(X, d)$ 
be an 
$R$-valued ultrametric space. 
Then,  
for every  fixed integer 
$n\in \zz_{\ge 1}$,  
the following statements are equivalent:
\begin{enumerate}[label=\textup{(A\arabic*)}]
\item\label{item:18:main11}
The space
$(X, d)$ 
is 
$(R, n)$-haloed. 

\item \label{item:18:main12}
The space
$(X, d)$ 
is 
$(R, n)$-avoidant. 

\item \label{item:18:main13}
The space
$(X, d)$ 
is 
$\youfin(R, n+1)$-injective. 

\end{enumerate}
As a consequence, 
 the following statements are equivalent:
\begin{enumerate}[label=\textup{(B\arabic*)}]
\item\label{item:18:main21}
The space
$(X, d)$ 
is 
$(R, \yoomega)$-haloed. 

\item\label{item:18:main22} 
The space
$(X, d)$ 
is 
$(R, \yoomega)$-avoidant. 

\item\label{item:18:main23}
The space
$(X, d)$ 
is 
$\youfin(R, \yoomega)$-injective. 

\end{enumerate}
\end{thm}

\begin{rmk}
Some of the  equivalences in 
Theorem 
\ref{thm:18:chara} 
have   been discovered. 
The equivalence between 
the conditions 
\ref{item:18:main22} 
and 
\ref{item:18:main23}
is 
 implicitly 
proven in 
\cite[Lemma 2.48]{Ishiki2023const}. 
The proof of Key Lemma  in 
\cite{MR748978} 
indicates 
the
 equivalence 
between 
the conditions 
\ref{item:18:main21}
and 
\ref{item:18:main23}. 
In \cite[Lemma 1 in \S 2]{MR1619924}, 
it is proven that 
a metric space 
$(X, d)$ 
is 
$\youfin(R, \yoomega)$-injective 
if and only if 
for every 
$r\in R$, 
for every finite 
$r$-equidistant subset  
$A$ 
of 
$X$, 
there exists a point 
$y\in X$ 
with 
$y\not\in  A$ 
such that 
$A\sqcup\{y\}$
is 
$r$-equidistant. 
This means that 
the condition
\ref{item:18:main23}
is  equivalent to the condition that 
$(X, d)$ 
is 
$(R, n)$-haloed 
for all 
$n\in \mm$
(see Lemma \ref{lem:18:3equi}). 
\end{rmk}

We next explain applications of Theorem
\ref{thm:18:chara}. 
For metric spaces 
$(X, d)$
and 
$(Y, e)$, 
we define the 
\emph{$\ell^{\infty}$-product metric  $d\times _{\infty}e$ 
on 
$X\times Y$}
 by 
$(d\times_{\infty} e)((x, y), (u, v))
=d(x, u)\lor e(y, v)$. 
For 
$p\in [1, \infty)$, 
we define the 
\yoemph{$\ell^{p}$-product 
metric 
$d\times_{p} e$} 
by 
$(d\times_{p} e)((x, y), (u, v))
=(d(x, u)^{p}+e(y, v)^{p})^{1/p}$.

First we prove that 
the Cartesian product preserves the  structure of  
the Urysohn universal ultrametric spaces. 
\begin{thm}\label{thm:18:product}
Let 
$R$ 
be a range set. 
If
$(X, d)$ 
is 
the 
$R$-Urysohn universal 
ultrametric space, 
then 
so is 
$(X\times X, d\times_{\infty}d)$. 
In particular, 
the space 
$(X\times X, d\times_{\infty}d)$
is isometric to 
$(X, d)$. 
\end{thm}

In contrast,  
we observe that the Cartesian product of 
the ordinary Urysohn universal metric space 
$(\yourysp, \yourydis)$ 
is 
not isometric to 
$(\yourysp, \yourydis)$.  
  \begin{thm}\label{thm:18:nonury}
  For 
  any 
  $p\in [1, \infty]$, 
 the product  space 
 $(\yourysp\times \yourysp, \yourydis\times_{p}\yourydis)$
  is not 
  $\yofin$-injective. 
  In particular, 
  it is not isometric to 
  $(\yourysp, \yourydis)$. 
  \end{thm}

For a metric space 
$(X, d)$ 
and for a subset of 
$X$, 
we 
 write 
$\yodiam(A)$ 
as the diameter of 
$A$. 
We denote by 
$\yoexpsp{X}$
the set of all non-empty compact subsets of 
$X$. 
For 
$m\in \zz_{\ge 2}\sqcup\{\infty\}$ 
and 
$l\in (0, \infty]$, 
we also denote by 
$\yoexpspml{X}{m}{l}$ 
the space 
of all 
$E\in \yoexpsp{X}$ 
such that 
$\card(E)\le m$
 and
  $\yodiam(E)\le l$, 
where 
in the case of  $m=l=\infty$, 
the inequalities 
$\card(E)\le \infty$ 
and 
$\yodiam(E)\le \infty$
 mean
 that 
there are no restrictions on  
$\card(E)$ 
and 
$\yodiam(E)$. 
In particular, 
$\yoexpspml{X}{\infty}{\infty}=\yoexpsp{X}$. 
For 
$E, F\in \yoexpsp{X}$, 
we define the Hausdorff distance 
$\yoexpdis{d}(E, F)$ 
by 
\[
\yoexpdis{d}(E, F)=
\max\left\{\sup_{a\in E}d(a, F), \sup_{b\in F}d(b, E)\right\}.
\] 
Notice that 
$\yoexpdis{d}$
is actually a metric on 
$\yoexpsp{X}$. 
We represent  the restricted metric 
$\yoexpdis{d}$ 
on 
$\yoexpspml{X}{m}{l}$
as the same symbol
$\yoexpdis{d}$ 
to the ambient metric. 
Remark that the space 
$(\yoexpspml{X}{m}{\infty}, \yoexpdis{d})$ 
is 
sometimes called the  
\yoemph{$m$-th symmetric product of 
$(X, d)$}. 

We next  
show that the hyperspace keeps the 
structure of 
the Urysohn universal ultrametric spaces.

\begin{thm}\label{thm:18:hyper}
Let 
$m\in \zz_{\ge 2}\sqcup \{\infty\}$,   
$R$ 
be a range set, 
$l \in (0, \infty]$, 
and 
let 
$(X, d)$ 
be the 
$R$-Urysohn 
universal 
ultrametric  space. 
Then 
$(\yoexpspml{X}{m}{l}, \yoexpdis{d})$ 
is 
the 
$R$-Urysohn 
universal  ultrametric space. 
In particular, 
the space 
$(\yoexpspml{X}{m}{l}, \yoexpdis{d})$  
is 
isometric to 
$(X, d)$. 
\end{thm}

We also establish   that  for every range set 
$R$, 
every complete 
$\youfin(R)$-injective 
ultrametric space 
contains a subspace isometric to 
the 
$R$-Urysohn universal ultrametric space. 
\begin{thm}\label{thm:18:ram}
Let 
$R$
 be an uncountable  range set. 
If 
 $(X, d)$ 
 is  a complete 
$\youfin(R, \yoomega)$-injective 
ultrametric space, 
then 
$(X, d)$ 
contains a subspace 
$F$ 
isometric to 
the 
$R$-Urysohn universal ultrametric space. 
Moreover, 
if
 $K$ 
 is 
a compact  subset  of 
$X$ 
satisfying  
$d(K^{2})\yosub R$, 
then 
we can 
choose 
$F$ 
so that 
$K\yosub  F$. 
\end{thm}

\begin{rmk}
Theorem 
\ref{thm:18:ram}
in the case where 
 $R$ 
 is 
finite or  countable
and 
an analogue of 
Theorem 
\ref{thm:18:ram}
for 
$(\yourysp, \yourydis)$
 can be proven  by 
 the
back-and-forth argument. 
\end{rmk}

This paper is organized  as follows. 
In Section 
\ref{sec:pre}, 
we prepare basic notations and notions. 
We will explain  the petaloid ultrametric space. 
In Section 
\ref{sec:chara}, 
 we characterize the
  $\youfin(R)$-injectivity 
using the existence of 
infinite  equidistant subsets of 
closed balls
(see Theorem \ref{thm:18:chara}). 
We divide the proof of Theorem 
\ref{thm:18:chara}
into three parts 
(Lemmas \ref{lem:A1ToA2}, 
\ref{lem:A2ToA3}, 
and 
\ref{lem:A3ToA1}). 
As a consequence, 
we prove Theorem 
\ref{thm:18:product} on 
the product of Urysohn universal ultrametric spaces. 
In Section 
\ref{sec:hyper},
we clarify the metric structures
 of  
 the hyperspace of compact subsets 
 and symmetric products  of 
 the 
 $R$-Urysohn universal ultrametric space
(see Theorem \ref{thm:18:hyper}). 
Section
 \ref{sec:existence} is devoted to 
proving that 
all 
complete
$\youfin(R)$-injective  
ultrametric spaces  contain the 
$R$-Urysohn universal ultrametric space
(see Theorem \ref{thm:18:ram}). 
Section 
\ref{sec:ques}
 presents some questions on the isometry problems of 
 the product and hyperspaces of 
 universal metric spaces
 such as the Urysohn universal
 metric  space and the random graph.

\section{Preliminaries}\label{sec:pre}
\subsection{Generalities}

The proof of the  next lemma
 is  presented 
 in 
\cite[Propositions
18.2, 
 18.4 and 18.5]{MR2444734}. 
\begin{lem}\label{lem:ultraopcl}
Let 
$(X, d)$ 
be an ultrametric space. 
Then the following are true:
\begin{enumerate}[label=\textup{(\arabic*)}]

\item\label{item:18:ball0}
For every triple 
$x, y, z\in X$, 
the inequality 
$d(x, z)<d(z, y)$ 
implies 
$d(z, y)=d(x, y)$.

\item\label{item:18:ball1}
For every 
$a\in X$, 
for 
every 
$r\in (0, \infty)$, 
and for every 
$q\in B(a, r)$ 
we have 
$\yocball(a, r)=\yocball(q, r)$.

\item\label{item:18:ball2}
For every pair 
 $a, b\in X$,
 and for every pair 
   $r, l\in (0, \infty)$, 
if 
$\yocball(a, r)\cap \yocball(b, l)\neq \emptyset$, 
then 
we have 
$\yocball(a, r)\yosub \yocball(b, l)$ 
or 
$\yocball(b, l)\yosub \yocball(a, r)$. 

\end{enumerate}
\end{lem}

\subsection{Urysohn universal ultrametric spaces}
A subset 
$E$ 
of 
$[0, \infty)$ 
is said to be
\yoemph{semi-sporadic} 
if 
there exists a strictly decreasing sequence 
$\{a_{i}\}_{i\in \zz_{\ge 0}}$ 
in 
$(0, \infty)$ 
such that 
$\lim_{i\to \infty}a_{i}=0$ 
and 
$E=\{0\}\cup \{\, a_{i}\mid i\in \zz_{\ge 0}\, \}$. 
A subset of 
$[0, \infty)$ 
is 
said to be
\yoemph{tenuous} 
if it is finite or semi-sporadic
(see 
\cite{Ishiki2023const}). 
For a range set 
$R$, 
we denote by 
$\yoofam{R}$
the set of all tenuous 
range subsets of 
$R$. 
For a metric space 
$(X, d)$, 
and for a subset 
$A$ 
of 
$X$, 
we write 
$d(x, A)=\inf_{a\in A}d(x, a)$. 
In this paper, 
we often represent a restricted metric 
as the same symbol to the ambient one. 
We now provide the definition of 
the petaloid ultrametric space that is introduced in 
\cite{Ishiki2023Ury} by the author.

\begin{df}\label{df:18:petaloid}
Let 
$R$ 
be an uncountable range set.  
We say that a metric space  
$(X, d)$ 
is 
\yoemph{$R$-petaloid} 
if 
it is an 
$R$-valued ultrametric space 
and 
there exists a family 
$\{\yopetal{X}{S}\}_{S\in \yoofam{R}}$
of subspaces of 
$X$
satisfying the following properties:
\begin{enumerate}[label=\textup{(P\arabic*)}]

\item\label{item:pr:sep}
If 
$S\in \yoofam{R}$ 
satisfies 
$S\neq \{0\}$, 
then 
 $(\yopetal{X}{S}, d)$
 is 
isometric to  the 
$S$-Urysohn 
universal ultrametric space. 
Namely, 
$(\yopetal{X}{S}, d)$
is a
separable  complete 
$\youfin(S)$-injective 
$S$-valued ultrametric space.

\item\label{item:pr:cup}
We have 
$\bigcup_{S\in \yoofam{R}}\yopetal{X}{S}=X$.

\item\label{item:pr:cap}
If 
$S,  T\in \yoofam{R}$, 
then
$\yopetal{X}{S}\cap \yopetal{X}{T}=\yopetal{X}{S\cap T}$.

\item\label{item:pr:distance}
If 
$S, T\in \yoofam{R}$ 
and 
$x\in \yopetal{X}{T}$, 
then
$d(x, \yopetal{X}{S})$ 
belongs to 
$\{0\}\cup (T\setminus S)$.

\end{enumerate}
We call the family 
$\{\yopetal{X}{S}\}_{S\in \yoofam{R}}$ 
an 
\yoemph{$R$-petal of $X$}, 
and call
$\yopetal{X}{S}$ 
 the 
 $S$-piece 
 of the 
 $R$-petal 
$\{\yopetal{X}{S}\}_{S\in \yoofam{R}}$. 
We simply write  
$\yopetalq{S}=\yopetal{X}{S}$
when the whole  space  is clear by the context. 
\end{df}

\begin{rmk}
For every uncountable range set 
$R$, 
and for  every  
$R$-petaloid space 
$(X, d)$, 
the petal
 $\yopetal{X}{\{0\}}$ 
 is 
a singleton. 
\end{rmk}

\begin{lem}\label{lem:18:p6}
Let 
$R$ 
be a range set, 
and 
$(X, d)$ 
be the 
$R$-petaloid ultrametric space. 
Then 
for every 
$S\in \yoofam{R}$,  
and 
for every 
$x\in X$, 
there exists 
$p\in \yopetal{X}{S}$ 
such that 
$d(x, \yopetal{X}{S})=d(x, p)$. 
\end{lem}
\begin{proof}
The lemma is  deduced from  
the property that 
$\yopetal{X}{S}$ 
is the 
$S$-Urysohn 
universal ultrametric space
(see the property \ref{item:pr:sep} 
and 
\cite[Propositions  20.2 and  21.1]{MR2444734}). 
\end{proof}

\begin{thm}\label{thm:18:petapeta}
Let 
$R$ 
be 
an 
uncountable range set. 
The following statements hold:
\begin{enumerate}[label=\textup{(\arabic*)}]

\item\label{item:uni}
There exists an 
$R$-petaloid 
ultrametric space
and it is unique up to isometry.

\item 
The 
$R$-petaloid 
ultrametric space is complete, 
non-separable, 
and 
$\yofin(R)$-injective.

\item 
Every separable
 $R$-valued 
 ultrametric space 
can be isometrically embedded into 
the 
$R$-petaloid ultrametric space. 
\end{enumerate}
\end{thm}
\begin{proof}
The first statement  is due 
to the author's result
\cite[Theorem 1.1]{Ishiki2023Ury}.
In the second statement, 
the 
$\youfin(R)$-injectivity 
and 
the completeness are deduced  from 
\cite[Propositions 2.9 and  2.10]{Ishiki2023Ury}. 
The non-separability is follows from 
\cite[(12) in Theorem 1.6]{MR3782290} 
and 
the fact 
that if 
$(X, d)$ 
is 
$\youfin(R)$-injective, 
then 
$d(X\times X)=R$
(see also Section 
\ref{sec:intro}).
All complete 
$\youfin(R)$-injective 
ultrametric spaces 
have the property stated in 
the third statement, 
which can be shown 
by induction. 
\end{proof}

Recall that 
$\yoomega=\zz_{\ge 0}$ 
as a set.

\begin{exam}\label{exam:gsp}
We give an example of 
the 
$R$-petaloid space. 
Let 
$R$ be a range set. 
We also denote by 
$\yomapsco{R}{\yoomega}$ 
the set of all function 
$f\colon R\to \yoomega$ 
such that 
$f(0)=0$ 
and 
the set 
$\{0\}\cup \{\, x\in R \mid f(x)\neq 0\, \}$ 
is tenuous. 
For 
$f, g\in \yomapsco{R}{\yoomega}$, 
we define an 
$R$-valued ultrametric 
$\yomaindisco$ 
on 
$\yomapsco{R}{\yoomega}$ 
by 
$\yomaindisco(f, g)=
\max\{\, r\in R\mid f(r)\neq g(r)\, \}$ 
if 
$f\neq g$; 
otherwise, 
$\yomaindisco(f, g)=0$. 
Then the space 
$(\yomapsco{R}{\yoomega}, \yomaindisco)$
is  
$R$-petaloid, 
and hence all 
$R$-petaloid 
ultrametric spaces are isometric to 
$(\yomapsco{R}{\yoomega}, \yomaindisco)$
(see 
\cite[Theorem 1.3]{Ishiki2023Ury}). 
For more information of this construction, 
we refer the readers to 
\cite{MR2754373}, 
\cite{MR2667917}, 
and \cite{delhomme2015homogeneous}. 
\end{exam}

\begin{rmk}
Even 
if 
$R$ 
is finite or countable, 
the 
$R$-Urysohn universal ultrametric space 
satisfies all of 
 the properties
\ref{item:pr:sep}--\ref{item:pr:distance}
  in the definition of the 
  $R$-petaloid space
  (Definiton \ref{df:18:petaloid}). 
Thus, 
the concept of the petaloid spaces can be
naturally considered as a generalization of 
separable Urysohn universal ultrametric spaces. 
\end{rmk}

\section{Characterizations of injectivity}\label{sec:chara}
In this section, 
we provide characterizations of 
the injectivity  for finite ultrametric spaces.

 \begin{lem}\label{lem:18:3equi}
Let 
 $R$ 
 be a range set. 
 Then the following statements are true:
 \begin{enumerate}[label=\textup{(\arabic*)}]

 \item\label{item:1:ninf}
 An ultrametric  space
$(X, d)$ 
is 
$(R, \yoomega)$-haloed 
if and only if 
$(X, d)$ 
is 
$(R, n)$-haloed 
for all 
$n\in \zz_{\ge 1}$.

\item\label{item:2:ninf}
An  ultrametric  space
$(X, d)$ 
is 
$(R, \yoomega)$-avoidant
if and only if 
$(X, d)$ 
is 
$(R, n)$-avoidant
for all 
$n\in \zz_{\ge 1}$.

\item\label{item:3:ninf}
An  
ultrametric  space
$(X, d)$ 
is 
$\youfin(R, \yoomega)$-injective
if and only if 
$(X, d)$ 
is 
$\youfin(R, n)$-injective
for all 
$n\in \zz_{\ge 1}$.

 \end{enumerate}
 \end{lem}
 \begin{proof}
 The statements 
 \ref{item:2:ninf} 
 and 
 \ref{item:3:ninf} 
 follow from 
 Definitions 
 \ref{df:18:halo} 
 and 
 \ref{df:18:avoidant}. 
Next we prove 
\ref{item:1:ninf}. 
We  only need  to show that 
if 
$(X, d)$ 
is 
$(R, n)$-haloed 
for all
 $n\in \zz_{\ge 1}$, 
 then it is 
$\youfin(R, \yoomega)$. 
Take 
$a\in X$ 
and 
$r\in R\setminus \{0\}$. 
We define a relation 
$\yorel$ 
on 
$\yocball(a, r)$ 
by declaring that 
$x\yorel y$ 
means 
$d(x, y)<r$. 
From the strong triangle inequality, 
it follows that 
$\yorel$ 
is an equivalence relation 
on 
$\yocball(a, r)$. 
Let
 $Q$ 
 denote the quotient set of 
 $\yocball(a, r)$ 
 by 
$\yorel$. 
Since 
$(X, d)$ 
is 
$(R, n)$-haloed 
for all 
$n\in \zz_{\ge 1}$, 
we notice that 
$n\le \card(Q)$ 
for all 
$n\in \zz_{\ge 1}$. 
Therefore 
$Q$ 
is infinite. 
Take a complete system 
$W$ 
of representatives of 
$Q$. 
Then 
$W$ 
becomes an infinite  
$r$-equidistant subset of 
$\yocball(a, r)$. 
This proves that 
$(X, d)$ 
is 
$\youfin(R, \yoomega)$-injective. 
 \end{proof}

The next three lemmas play  key roles to 
prove our first main result. 
\begin{lem}\label{lem:A1ToA2}
Let 
$R$
 be a range set and 
 $n\in \zz_{\ge 1}$. 
If an ultrametric space
 $(X, d)$
  is 
$(R, n)$-haloed, 
then it is 
$(R, n)$-avoidant. 
\end{lem}
\begin{proof}
To show that 
$(X, d)$ 
is
 $(R, n)$-avoidant, 
take
 $a\in X$, 
 $r\in R$ 
 and take 
a subset $A$ 
of 
$\yocball(a, r)$ 
with $\card(A)<n$. 
Since 
$(X, d)$ 
is 
$(R, n)$-haloed, 
we can find 
an 
$r$-equidistant  subset 
$H$ 
of 
$\yocball(a, r)$
with
 $n\le \card(H)$. 
Since 
$d$ 
is an ultrametric, 
for every 
$x\in A$, 
we obtain 
$\card(\{\, y\in H\mid d(x, y)<r\, \})\le 1$. 
Thus, 
due to 
$\card(A)<\card(H)$,  
we can take 
$p\in H$
 such that 
$r\le d(p, x)$ 
for all 
$x\in A$. 
In this situation,  
the statement 
\ref{item:18:ball1} in Lemma 
\ref{lem:ultraopcl} implies that 
$d(x, p)=r$ 
for all 
$x\in A$. 
\end{proof}

\begin{lem}\label{lem:A2ToA3}
Let 
$R$ 
be a range set and 
$n\in \zz_{\ge 1}$. 
If an  ultrametric space 
$(X, d)$ 
is 
$(R, n)$-avoidant, 
then it is 
$\youfin(R, n+1)$-injective. 
\end{lem}
\begin{proof}
Take an arbitrary 
 $R$-valued ultrametric space 
$(Y\sqcup\{\theta\}, e)$ 
in 
$\youfin(R, n+1)$,  
and 
take  an isometric embedding 
$\phi\colon Y\to X$. 
Notice that 
$\card(Y\sqcup\{\theta\})<n+1$. 
To show that
$(X, d)$ is 
$\youfin(R, n+1)$-injective, 
we only  need to verify  that 
there exists an isometric embedding 
$\Phi\colon Y\sqcup\{\theta\}
\to X$
with 
$\Phi|_{Y}=\phi$. 
Put 
$r=\min_{y\in Y}e(y, \theta)$
 and take 
$q\in Y$ 
such that:
\begin{enumerate}[label=\textup{(\roman*)}, series=SHIKI]

\item\label{item:reqt} 
$r=e(q, \theta)$. 
\end{enumerate}
In this situation, we have 
$r\in R$. 
Put 
$A=\phi(Y)\cap \yocball(\phi(q), r)$. 
Notice  that 
$\card(A)<n$. 
Since 
$(X, d)$ 
is 
$(R, n)$-avoidant 
and since 
$\card(A)< n$, 
there exists 
$t\in \yocball(\phi(q), r)$ 
such that: 
\begin{enumerate}[label=\textup{(\roman*)}, resume*=SHIKI]

\item\label{item:datr}
$d(a, t)=r$
 for all 
$a\in A$. 
\end{enumerate}
We now prove the following statement. 
\begin{enumerate}[label=\textup{(A)}]

\item\label{item:18:condA}
We have 
$d(\phi(y), t)=e(y, \theta)$
for all 
$y\in Y$. 
\end{enumerate}
We divide the 
proof of the statement 
\ref{item:18:condA} into two cases. 

\yocase{1}{$e(y, q)\le r$}
Then 
$e(y, \theta)\le e(y, q)\lor e(q, \theta)\le r$. 
By the minimality of 
$r$, 
we conclude that 
$e(y, \theta)=r$. 
From 
\ref{item:datr}, 
it follows that 
$d(\phi(y), t)=r$. 
Thus 
$d(\phi(y), t)=e(y, \theta)$.

\yocase{2}{$r< e(y, q)$}
In this case, 
the equality 
\ref{item:reqt} 
yields
$e(q, \theta)<e(y, q)$. 
Using  
\ref{item:18:ball0} 
in 
Lemma \ref{lem:ultraopcl}, 
we have 
\begin{enumerate}[label=\textup{(\roman*)}, resume*=SHIKI]
\item\label{item:eyteyq} 
$e(y, \theta)=e(y, q)$. 
\end{enumerate}
Since
 $\phi$ 
 is isometric, 
we obtain
\begin{enumerate}[label=\textup{(\roman*)}, resume*=SHIKI]
\item\label{item:eyqdpydq}
$e(y, q)=d(\phi(y), \phi(q))$. 
\end{enumerate}
The inequality  
$r<e(y, q)$ 
and 
\ref{item:eyqdpydq} 
imply $r<d(\phi(y), \phi(q))$, 
which means 
that 
$\phi(y)\not\in \yocball(\phi(q), r)$. 
From 
 \ref{item:18:ball1} in 
 Lemma 
 \ref{lem:ultraopcl}, 
 we deduce that 
$\yocball(\phi(q), r)=\yocball(t, r)$. 
Hence 
$\phi(y)\not \in \yocball(t, r)$. 
Thus 
$r<d(\phi(y), t)$. 
Using 
$r=d(\phi(q), t)$
(see the definition of 
$t$ 
and 
the eqaulity
\ref{item:datr}), 
we 
have the 
inequality
$d(\phi(q), t)<d(\phi(y), t)$, 
and hence
the statement  
\ref{item:18:ball0} in Lemma \ref{lem:ultraopcl} 
implies that:
\begin{enumerate}[label=\textup{(\roman*)}, resume*=SHIKI]

\item\label{item:dpytdpyq}
$d(\phi(y), \phi(q))=d(\phi(y), t)$. 
\end{enumerate}
Combining 
 equalities 
\ref{item:eyteyq}, 
\ref{item:eyqdpydq}, 
and 
\ref{item:dpytdpyq}, 
we conclude that 
$d(\phi(y), t)=e(y, \theta)$. 
Then the  proof of the statement 
\ref{item:18:condA} 
is completed. 

Next we define a map  
$\Phi\colon Y\sqcup\{\theta\}\to X$ 
by 
$\Phi(x)=\phi(x)$
 for all 
 $x\in Y$ 
 and 
$\Phi(\theta)=t$. 
The statement 
\ref{item:18:condA} shows that 
 $\Phi$ 
 is an isometric embedding 
with 
$\Phi|_{Y}=\phi$. 
Therefore 
the space
 $(X, d)$ 
 is 
$\youfin(R, n+1)$-injective. 
\end{proof}

\begin{lem}\label{lem:A3ToA1}
Let 
$R$ 
be a range set and 
$n\in \zz_{\ge 1}$. 
If an ultrametric space 
$(X, d)$ 
is 
$\youfin(R, n+1)$-injective, 
then it is 
$(R, n)$-haloed. 
\end{lem}
\begin{proof}
Take 
$a\in X$,  
and
 $r\in R\setminus \{0\}$. 
By recursion,  
we will construct 
$\{A_{i}\}_{i=1}^{n}$ 
such that:
\begin{enumerate}

\item 
we have 
$A_{i}\yosub \yocball(a, r)$ 
for all 
$i\le n$;

\item 
for every 
$i\in \{1, \dots, n\}$, 
we have 
$\card(A_{i})=i$;

\item 
for every 
$i\in \{1, \dots, n-1\}$, 
we  have 
$A_{i}\yosub A_{i+1}$;

\item 
each 
$A_{i}$ 
is 
$r$-equidistant. 
\end{enumerate}
Take 
$s\in \yocball(a, r)$ 
and put 
$A_{1}=\{s\}$. 
Fix 
$l< n$ 
and  
 assume that we have already constructed 
 $\{A_{i}\}_{i=1}^{l}$. 
 We shall construct 
 $A_{l+1}$. 
 Take a point 
 $\theta$ 
 such that 
  $\theta\not \in X$
  and 
define an ultrametric 
$e$ 
on 
$A_{l}\sqcup \{\theta\}$ by 
$e|_{A_{l}^{2}}=d$ 
and 
$d(x, \theta)=r$
for all 
$x\in A_{l}$. 
We also define 
an isometric embedding 
$\phi\colon A_{l}\to X$ 
by 
$\phi(a)=a$. 
Using the 
$\youfin(R, n+1)$-injectivity
together with  
$\card(A_{l}\sqcup\{\theta\})<n+1$, 
we obtain 
$p\in \yocball(a, r)$ 
such that 
$d(\phi(x), p)=e(x, \theta)=r$
for all 
$x\in A_{l}$. 
Put 
$A_{l+1}=A_{l}\sqcup \{p\}$. 
Then the resulting  set 
$A_{n}$ 
is 
$r$-equidistant and 
satisfies 
$\card(A_{n})=n$. 
Hence
 $(X, d)$ 
 is 
 $(R, n)$-haloed. 
\end{proof}
Now we are able to  prove Theorem 
\ref{thm:18:chara}. 
\begin{proof}[Proof of Theorem \ref{thm:18:chara}]
Combining Lemmas 
\ref{lem:A1ToA2}, 
\ref{lem:A2ToA3}, 
and 
\ref{lem:A3ToA1}, 
we obtain the implications 
$\ref{item:18:main11}\To
\ref{item:18:main12}$, 
$\ref{item:18:main12}\To 
\ref{item:18:main13}$, 
and 
$\ref{item:18:main13}\To 
\ref{item:18:main11}$. 
Thus the three statements 
\ref{item:18:main11}, 
\ref{item:18:main12}, 
and 
\ref{item:18:main13} are 
equivalent to each other.

 The 
 equivalences between 
\ref{item:18:main21}, 
\ref{item:18:main22}, and 
\ref{item:18:main23}  
are
 deduced from 
 Lemma 
\ref{lem:18:3equi} and 
the former part
of the theorem. 
This finishes the proof
of Theorem 
\ref{thm:18:chara}. 
\end{proof}

Using Theorem 
\ref{thm:18:chara}, 
we clarify the metric structure of 
the Cartesian product of 
the Urysohn universal ultrametric spaces. 

\begin{lem}\label{lem:18:prodhalo}
If 
$R$ 
is a range set
and 
an ultrametric space 
$(X, d)$ 
is 
$(R, \yoomega)$-haloed,
then the space 
$(X\times X, d\times_{\infty}d)$ 
is 
$(R, \yoomega)$-haloed. 
\end{lem}
\begin{proof}
To prove that 
$(X\times X, d\times_{\infty}d)$ 
is 
$(R, \yoomega)$-haloed, 
take 
$(a, b)\in X\times X$ 
and 
$r\in R\setminus\{0\}$. 
Since $(X, d)$ 
is 
$(R, \yoomega)$-haloed, 
we can find countable 
$r$-equidistant subsets 
$E$ 
and 
$F$ 
of 
$\yocball(a, r)$ 
and 
$\yocball(b, r)$, 
respectively. 
By the definition of 
$d\times_{\infty} d$, 
the set 
$E\times F$ 
is 
a countable 
$r$-equidistant subset 
of 
$\yocball((a, b), r; d\times_{\infty}d)$. 
\end{proof}

Subsequently, 
we show Theorem 
\ref{thm:18:product}. 
\begin{proof}[Proof of Theorem \ref{thm:18:product}]
We divide the 
proof into the countable and 
uncountable cases. 

\yocase{1}{$R$ is finite or  countable}
In this case, 
 $(X, d)$ 
 is 
 a separable complete 
 $\youfin(R, \yoomega)$-injective 
 $R$-valued ultrametric space. 
In particular, 
 $(X\times X, d\times_{\infty}d)$
 is separable and complete. 
Lemma 
\ref{lem:18:prodhalo} 
shows that 
$(X\times X, d\times_{\infty}d)$ 
is 
$(R, \yoomega)$-haloed, 
and hence it is 
$\youfin(R, \yoomega)$-injective 
by 
Theorem \ref{thm:18:chara}. 
Thus 
the product 
$(X\times X, d\times_{\infty}d)$
is the 
$R$-Urysohn
 universal 
ultrametric space.

\yocase{2}{$R$ is uncountable}
In this setting, 
$(X, d)$ 
is the 
$R$-petaloid ultrametric space 
(see Definition \ref{df:18:origin}). 
We define a petal structure on 
$(X\times X, d\times_{\infty}d)$
by 
$\yopetal{X\times X}{S}=\yopetal{X}{S}\times \yopetal{X}{S}$
for 
$S\in \yoofam{R}$. 
This structure satisfies the 
property
\ref{item:pr:cap}. 
The property 
\ref{item:pr:sep}
follows from 
the separable case proven above. 
We now  verify 
\ref{item:pr:cup}. 
For every pair 
$S, T\in \yoofam{R}$, 
we have 
$\yopetal{X}{S}\cup \yopetal{X}{T}\yosub 
\yopetal{X}{S\cup T}$
(see \cite[Lemma 2.1]{Ishiki2023Ury}). 
Then  
we also have 
$\yopetal{X}{S}\times \yopetal{X}{T}\yosub \yopetal{X}{S\cup T}\times 
\yopetal{X}{S\cup T}$. 
Thus,  
using 
$X=\bigcup_{S\in \yoofam{R}}\yopetal{X}{S}$
(the property 
\ref{item:pr:cup} 
for 
$(X, d)$), 
we obtain 
\[
X\times X=\bigcup_{S, T\in \yoofam{R}}\yopetal{X}{S}\times 
\yopetal{X}{T}=\bigcup_{S\in \yoofam{R}}\yopetal{X\times X}{S}. 
\]
This means that 
\ref{item:pr:cup} 
is valid. 
Next we confirm the property 
\ref{item:pr:distance}. 
Take $S, T\in \yoofam{R}$
and take 
$(x, y)\in \yopetal{X\times X}{T}$. 
From   $\yopetal{X\times X}{S}
=\yopetal{X}{S}\times \yopetal{X}{S}$, 
and 
the property 
\ref{item:pr:distance} 
for 
$(X, d)$, 
it follows that 
$d(x, \yopetal{X}{S}), d(y, \yopetal{X}{S})
\in \{0\}\cup (T\setminus S)$. 
Thus the 
 the definition of 
$d\times_{\infty}d$
yields 
\[
(d\times_{\infty}d)((x, y), \yopetal{X\times X}{S})=
\max\{d(x, \yopetal{X}{S}), d(y, \yopetal{X}{S})\}. 
\]
Hence 
$(d\times_{\infty}d)((x, y), \yopetal{X\times X}{S})\in \{0\}\cup (T\setminus S)$. 
Namely, the property 
\ref{item:pr:distance} 
is fullfiled. 

Therefore 
the product 
$(X\times X, d\times_{\infty}d)$ 
is 
the 
$R$-petaloid space. 
The proof of Theorem 
\ref{thm:18:product} is finished. 
\end{proof}

\begin{rmk}
Similarly to 
Theorem 
\ref{thm:18:product}, 
we can prove the next statement:
For every pair of  range sets 
$R_{0}$ 
and 
$R_{1}$, 
let 
$(X, d)$ 
and 
$(Y, e)$ 
denote the 
$R_{0}$-Urysohn 
and 
the 
$R_{1}$-Urysohn 
universal ultrametric spaces, respectively. 
Then the product 
$(X\times Y, d\times_{\infty}e)$ is 
the 
$(R_{0}\cup R_{1})$-Urysohn
 universal ultrametric space. 
\end{rmk}


We say that 
a metric space 
$(X, d)$ 
is 
\emph{(isometrically)
homogeneous}
if for every pair $x, y\in X$, 
there exists an isometric bijection 
$f\colon X\to X$ such that 
$f(x)=y$. 

As an application, 
we obtain the next proposition. 

\begin{prop}
Let $R$ be a range set, 
$n\in \zz_{\ge 1}$, 
and 
$(X, d)$
be an
$R$-valued ultrametric space. 
Assume that 
the following conditions: 
\begin{enumerate}[label=\textup{(\arabic*)}]
\item 
there exists $a\in X$ such that 
for every $r\in R\setminus \{0\}$
there exists an $r$-equidistant  subset $A$ of 
$\yocball(a, r)$ with $\card(A)\ge n$;
\item 
$(X, d)$ is isometrically homogeneous. 
\end{enumerate}
Then 
$(X, d)$
is $\youfin(R, n)$-injective. 
\end{prop}

\begin{rmk}
In 
\cite{delhomme2015homogeneous}, 
it  is 
shown  that 
 an ultrametric space
 $(X, d)$ is isometrically homogeneous 
 if and only if it is 
 \emph{ultrahomogeneous}, 
 i.e., 
 for every finite subset $A$ of $X$ and 
 for every isometric embedding 
 $I\colon A\to X$, there exists 
 a bijective isometry 
 $H\colon X\to X$ such that 
 $H|_{A}=I$. 
 
\end{rmk}

\begin{lem}\label{lem:triinq}
Let 
$X$
 be a set, 
 and 
 $w\colon X\times X\to [0, \infty)$
  be a symmetric function, 
  i.e., 
  $w(x, y)=w(y, x)$
  whenever 
  $x, y\in X$.  
If there exists $L\in (0, \infty)$ such that 
$w(x, y)\in [L, 2L]$ for all distinct $x, y\in X$, 
then $w$
satisfies the 
triangle inequality. 
\end{lem}
\begin{proof}
Take mutually distinct three points  $x, y, z\in X$. 
Then we have 
\[
w(x, y)\le 2L=L+L\le w(x, z)+w(z, y). 
\]
This completes the proof. 
\end{proof}

The proof of Theorem 
\ref{thm:18:nonury}  
is 
based on linear algebra. 
  \begin{proof}[Proof of Theorem \ref{thm:18:nonury}]
Take
 $s_{0}, s_{1}\in \yourysp$
  with 
  $\yourydis(s_{0}, s_{1})=2^{-1/p}$ if $p<\infty$; 
  otherwise,  
$\yourydis(s_{0}, s_{1})=1$.
Put  
$A=\{s_{0}, s_{1}\}\times \{s_{0}, s_{1}\}$. 
In other words, 
$A=\{(s_{0}, s_{0}), (s_{0}, s_{1}), (s_{1}, s_{0}), (s_{1}, s_{1})\}$. 
Of course,
 $A$ 
 is a subset of 
 $\yourysp\times \yourysp$. 
  We now consider a one-point extension of 
  $(A, \yourydis\times_{p}\yourydis)$. 
  Let 
  $\theta$ 
  be a point such that 
  $\theta\not\in A$. 
For each 
$\mathbf{r}=(r_{00}, r_{01}, r_{10}, r_{11})\in [2^{-1}, 1]^{4}$, 
we define a metric 
$e_{\mathbf{r}}$ 
on  
$A\sqcup\{\theta\}$ 
by 
$e_{\mathbf{r}}|_{A^{2}}=\yourydis\times_{p}\yourydis$ and 
  $e_{\mathbf{r}}((s_{i}, s_{j}), \theta)=r_{ij}$. 
  Since each 
  $r_{ij}$
   belongs to 
   $[2^{-1}, 1]$ 
   and 
  $\yourydis\times_{p}\yourydis$ 
  on 
  $A$ 
  takes only values in 
  $\{0, , 2^{-1/p}, 1\}$
  and since 
  $\{2^{-1/p}, 1\}\yosub [2^{-1}, 1]$, 
  the function 
  $e_{\mathbf{r}}\colon (A\sqcup \{\theta\})^{2}\to \rr$
  actually satisfies the triangle inequality
  due to Lemma 
\ref{lem:triinq} for $L=2^{-1}$.

  To establish that 
  $(\yourysp\times \yourysp, \yourydis\times_{p}\yourydis)$ 
  is not 
  $\yofin$-injective, 
 we shall  show that we can take  a point  
  $\mathbf{r}\in [2^{-1}, 1]^{4}$ 
  such that 
  there is no point 
  $(\alpha, \beta)\in \yourysp\times \yourysp$ 
  satisfying that:
  \begin{enumerate}[label=\textup{(I)}]
  \item\label{item:18:inj}
  $e_{\mathbf{r}}((s_{i}, s_{j}), \theta)=(\yourydis\times_{p}\yourydis)((s_{i}, s_{j}), (\alpha, \beta))$
  for all 
  $(s_{i}, s_{j})\in A$. 
  \end{enumerate}
 Suppose, 
 contrary to our claim,  
 that 
 for every 
 $\mathbf{r}\in [2^{-1}, 1]^{4}$, 
 there exists 
 $(\alpha, \beta)\in \yourysp\times \yourysp$ 
 satisfying the condition 
 \ref{item:18:inj}. 
  Put 
  $x=\yourydis(s_{0}, \alpha)$,
 $y=\yourydis(s_{1}, \alpha)$, 
 $z=\yourydis(s_{0}, \beta)$, 
 and 
 $w=\yourydis(s_{1}, \beta)$.

In the case of 
$p=\infty$, 
take
 $\mathbf{r}=(r_{00}, r_{01}, r_{10}, r_{11})\in [2^{-1}, 1]^{4}$ 
 so that 
$r_{01}<r_{11}$, 
and 
$r_{10}<r_{11}$. 
By the definition of 
$\yourydis\times_{\infty}\yourydis$, 
the values 
$x$, 
$y$, 
$z$, 
and 
$w$ should satisfy 
$x\lor z=r_{00}$, 
$x\lor w=r_{01}$,  
$y\lor z=r_{10}$, 
and 
$y\lor w=r_{11}$. 
Then we have  
$y\le r_{10}$ 
and 
$w\le r_{01}$, 
which  imply that 
$y\lor w\le r_{10}\lor r_{01}<r_{11}$. 
This is a contradiction to
 $y\lor w=r_{11}$. 
 
 In the case of 
 $p<\infty$, 
 by the definition of 
 $\yourydis\times_{p}\yourydis$
 and the condition 
 \ref{item:18:inj}, 
 the values 
 $x$, 
 $y$, 
 $z$, 
 and 
 $w$
 should satisfy 
 $x^{p}+z^{p}=r_{00}^{p}$, 
 $x^{p}+w^{p}=r_{01}^{p}$, 
 $y^{p}+z^{p}=r_{10}^{p}$, 
 and 
 $y^{p}+w^{p}=r_{11}^{p}$. 
 Namely, we obtain
 \begin{align}\label{al:18:matrix}
\begin{pmatrix}
1 & 0 & 1 & 0 \\
1 & 0& 0 & 1 \\
0 & 1 & 1 & 0 \\
0 & 1 & 0 & 1
\end{pmatrix}
\begin{pmatrix}
x^{p}\\
y^{p}\\
z^{p}\\
w^{p}
\end{pmatrix}
=
\begin{pmatrix}
r_{00}^{p}\\
r_{01}^{p}\\
r_{10}^{p}\\
r_{11}^{p}
\end{pmatrix}.
 \end{align}
Since
 $\mathbf{r}\in [2^{-1}, 1]^{4}$ 
 is arbitrary, 
we notice that  the set 
$[2^{-p}, 1]^{4}$ 
is 
contained in 
 the image of the matrix 
 in 
 \eqref{al:18:matrix}. 
 However, 
 the rank of the matrix is 
 three 
(since the dimension of the kernel is one). 
This result  contradicts the fact that 
the topological dimension of  
$[2^{-p}, 1]^{4}$ 
is four 
(or we can choose  four specific  linearly independent  vectors  from this set). 
This completes the proof of 
Theorem 
\ref{thm:18:nonury}. 
 \end{proof}
\section{Hyperspaces as universal spaces}\label{sec:hyper}
In this section, 
we investigate the metric structures of 
hyperspaces of
Urysohn 
 universal ultrametric spaces.

We begin with the following 
 well-known lemma on 
 the Hausdorff distance. 
\begin{lem}\label{lem:18:folklore1}
For every metric space 
$(X, d)$, 
and for  every pair 
 $ E, F\in \yoexpsp{X}$, 
the
Hausdorff distance
$\yoexpdis{d}(E, F)$ is 
equal to 
the infimum of all 
$r\in (0, \infty)$ 
such that 
$E\yosub \bigcup_{b\in F}\yocball(b, r)$ 
and 
$F\yosub \bigcup_{a\in E}\yocball(a, r)$. 
\end{lem}

For the sake of self-containedness, 
we give a proof of the next
basic 
lemma. 
\begin{lem}\label{lem:18:folklore2}
Take  
$m\in \zz_{\ge 2}\sqcup\{\infty\}$, 
and 
$l\in (0, \infty]$. 
Let 
$(X, d)$ 
be a metric space. 
Then the next two statements are true:
\begin{enumerate}[label=\textup{(\arabic*)}]
\item\label{item:18:expsep}
If 
$(X, d)$ 
is separable, then 
so is 
$(\yoexpspml{X}{m}{l}, \yoexpdis{d})$. 
\item\label{item:18:expcomp}
If 
$(X, d)$ 
is complete, 
then 
so is 
$(\yoexpspml{X}{m}{l}, \yoexpdis{d})$. 
\end{enumerate}
\end{lem}
\begin{proof}
To prove  
\ref{item:18:expsep}, 
take a countable dense subset 
$Q$ 
of 
$X$. 
Then the set of all finite  non-empty subsets of 
$Q$ 
is dense in 
$(\yoexpsp{X}, \yoexpdis{d})$.
Thus, it is separable. 
Since 
$\yoexpspml{X}{m}{l}$
 is a metric subspace of 
 $\yoexpsp{X}$, 
 the space 
 $(\yoexpspml{X}{m}{l}, \yoexpdis{d})$
 is separable.

Next we  show 
\ref{item:18:expcomp}. 
The completeness of 
$(\yoexpsp{X}, \yoexpdis{d})$ 
can be  proven  in 
a 
similar way to 
\cite[Proposition 7.3.7]{BBI}. 
Thus 
it suffices to verify that 
$\yoexpspml{X}{m}{l}$ 
is 
closed in 
$(\yoexpsp{X}, \yoexpdis{d})$
for all 
$m\in \zz_{\ge 2}$ 
and 
$l\in (0, \infty)$. 
Take a sequence 
$\{K_{i}\}_{i\in \mm}$ 
in 
$\yoexpspml{X}{m}{l}$
and assume that 
it  converges to 
a set 
$L\in \yoexpsp{X}$. 
Since the map 
$D\colon \yoexpsp{X}\to \rr$ 
defined by 
$D(E)=\yodiam E$ 
is continuous 
(in fact, 
it is $2$-Lipschitz),  
we obtain
$\yodiam L\le l$. 
We next 
 show 
 $\card(L)\le m$. 
For the sake of contradiction, 
suppose that 
$m<\card(L)$. 
Then there exists a subset 
$E$ 
of 
$L$ 
with 
$\card(E)=m+1$. 
Put 
$r=\min \{\, d(x, y)\mid x, y\in E, x\neq y\, \}$, 
and 
take a sufficiently large
 $n\in \mm$ 
so that 
$\yoexpdis{d}(L, K_{n})<r/3$. 
In this case, 
Lemma 
\ref{lem:18:folklore1} 
implies that 
for every 
$x\in E$, 
there exists 
$y_{x}\in K_{n}$
 such that 
 $d(x, y_{x})<r/3$. 
By the definition of 
$r$, 
we have 
$y_{x}\neq y_{x^{\prime}}$ 
for all 
distinct 
$x, x^{\prime}\in E$. 
Thus the set 
$F=\{\, y_{x}\in K_{n}\mid x\in E\, \}$ 
has 
cardinality 
$m+1$. 
This is impossible due to 
$\card(K_{n})\le m$. 
Therefore  
$\card(L)\le m$. 
Subsequently, 
we obtain 
$L\in \yoexpspml{X}{m}{l}$. 
This finishes the proof. 
\end{proof}

Let 
$(X, d)$ 
be an ultrametric space, 
and 
$r\in (0, \infty)$. 
For  
$E\in \yoexpsp{X}$
we denote by 
$\yobrasp{E}{r}$
the set 
$\{\, \yocball(a, r)\mid a\in E\, \}$. 
For 
 convenience, 
we also define  
$\yobrasp{E}{0}=\{\, \{a\}\mid a\in E\, \}$, 
which is the case of 
$r=0$.  
By 
\ref{item:18:ball2} in 
Lemma 
\ref{lem:ultraopcl}
and the compactness of 
$E$, 
the set 
$\yobrasp{E}{r}$ 
is finite for all 
$r\in (0, \infty)$. 
The next
proposition 
 is an analogue of 
\cite[Theorem 5.1]{MR4462868} 
and 
\cite[Corollary 2.30]{Ishiki2023const}
for the Hausdorff distance.

\begin{prop}\label{prop:18:haus}
Let 
$(X, d)$ 
be an ultrametric space. 
Then for 
every 
pair of  subsets 
$E$ and 
$F$  of 
$X$, 
the value 
$\yoexpdis{d}(E, F)$ 
is equal to the minimum 
$r\in R$ 
such that 
$\yobrasp{E}{r}=\yobrasp{F}{r}$. 
\end{prop}
\begin{proof}
Let 
$W$ 
denote the set of all 
$r\in R$ 
such that 
$\yobrasp{E}{r}=\yobrasp{F}{r}$. 
Put 
$u=\inf W$. 
First we shall prove that 
$u$ 
is equal to 
the 
Hausdorff distance between 
$E$ 
and 
$F$. 
For every 
$r\in W$, 
the equality 
$\yobrasp{E}{r}=\yobrasp{F}{r}$ 
is true. 
This result means that 
$\{\, \yocball(a, r)\mid a\in E\, \}=
\{\, \yocball(b, r)\mid b\in F\, \}$. 
This equality 
yields 
$E\yosub \bigcup_{b\in F}\yocball(b, r)$ 
and 
$F\yosub \bigcup_{a\in E}\yocball(a, r)$. 
Thus 
Lemma 
\ref{lem:18:folklore1} 
implies that 
$\yoexpdis{d}(A, B)\le u$. 
To obtain the opposite inequality, 
take an arbitrary number  
$l$ 
with  
$\yoexpdis{d}(A, B)<l$. 
Then 
$E\yosub \bigcup_{b\in F}\yocball(b, l)$
 and 
$F\yosub \bigcup_{a\in E}\yocball(a, l)$. 
To show 
$\yobrasp{E}{l}\yosub \yobrasp{F}{l}$, 
take 
$\yocball(s, l)\in \yobrasp{E}{l}$. 
From
 $E\yosub \bigcup_{b\in F}\yocball(b, l)$, 
 it follows that 
 there exists $b \in F$ such that 
 $s\in \yocball(b, l)$. 
Thus 
the statement
 \ref{item:18:ball2} 
 in Lemma 
\ref{lem:ultraopcl} 
yields
$\yocball(s, l)=\yocball(b, l)$. 
Namely, 
$\yocball(s, l)\in \yobrasp{F}{l}$, 
and hence 
$\yobrasp{E}{l}\yosub \yobrasp{F}{l}$. 
Replacing the role of 
$\yobrasp{E}{l}$
with that of 
$\yobrasp{F}{l}$, 
we also obtain 
$\yobrasp{F}{l}\yosub \yobrasp{E}{l}$.
As a result, we have 
$\yobrasp{E}{l}=\yobrasp{F}{l}$. 
Therefore  
$u\le \yoexpdis{d}(E, F)$, 
and we then obtain 
$u=\yoexpdis{d}(E, F)$.

Next we will show that 
$u=\min W$. 
If 
$u=0$, 
then 
$u\in R$
 and
$E=F$. 
Thus 
$\yobrasp{E}{0}=\yobrasp{F}{0}$. 
Namely, 
$u=\min W$. 
In what follows, 
we may assume that 
$u>0$.

Now we show that 
$u\in R$. 
Put 
$S=d(E^{2})\cup d(F^{2})$. 
Then 
$S$ 
is tenuous
(see 
\cite[Corollary 2.28]{Ishiki2023const})
and 
$S\yosub R$. 
To obtain a contradiction, 
suppose that 
$u\not \in S$. 
Then, 
since
 $S$ is tenuous,  
 we can take 
$s, t \in S\setminus \{0\}$
 such that 
$u\in (s, t)$ 
and 
$(s, t)\cap S=\emptyset$. 
In this setting, 
we obtain the following statement: 
\begin{enumerate}[label=\textup{(B)}]
\item\label{item:18:balleq}
For every
$x\in E\cup F$
and for every 
$r\in (s, t)$, 
we have 
$\yocball(x, r)=\yocball(x, s)$. 
\end{enumerate}
Since 
$u=\inf W$, 
there exists 
$v\in W$ 
such that 
$v\in [u, t)$. 
Due to 
$v\in W$, 
we have 
$\yobrasp{E}{v}=\yobrasp{F}{v}$. 
This equality and the statement 
\ref{item:18:balleq}
show that 
$\yobrasp{E}{s}=\yobrasp{F}{s}$, 
and hence 
$s\in W$, 
which is a contradiction to 
$u=\inf W$ 
and 
$s<u$. 
Therefore 
$u\in S$, 
and hence 
$u\in R$.

To verify  that 
$u$ 
attains the  minimum, 
we will show  
$u\in W$. 
Take a sequence
 $\{r_{i}\}_{i\in \mm}$ 
 in 
 $W$ 
 such that 
 $u\le r_{i}$ 
 for all 
 $i\in \mm$, 
and 
$\lim_{i\to \infty}r_{i}=u$. 
If 
$r_{i}=u$ 
for some 
$i\in \mm$, 
we have 
$u\in W$. 
We may assume that 
$r_{i}\neq u$ 
for any 
$i\in \mm$. 
Since 
$S$ 
is tenuous and 
$u\in S\setminus \{0\}$, 
we can take 
$q\in S$ 
such that 
$(u, q)\cap S=\emptyset$. 
For a sufficiently large 
$i\in \mm$, 
the value 
$r_{i}$ 
belongs to 
$(u, q)$. 
Using the same argument as the proof 
of 
 $u\in R$, 
similarly to the statement 
\ref{item:18:balleq}, 
we notice that every 
$x\in E\cup F$ 
satisfies 
$\yocball(x, r_{i})=\yocball(x, u)$. 
Consequently, 
we obtain 
$\yobrasp{E}{u}=\yobrasp{F}{u}$, 
which means that
$u\in W$. 
As a result, we conclude that 
$u$ 
is minimal in 
$W$. 
This completes the proof. 
\end{proof}

As a consequence of 
Proposition 
\ref{prop:18:haus}, 
we obtain the following
 well-known statement. 
\begin{cor}\label{cor:18:val}
Let 
$R$ 
be a range set and 
$(X, d)$ 
be an ultrametric. 
The the following
 two statements are true:
\begin{enumerate}[label=\textup{(\arabic*)}]

\item\label{item:18:ult}
The Hausdorff metric 
$\yoexpdis{d}$ 
is an 
ultrametric on 
$\yoexpsp{X}$.

\item\label{item:18:rval}
If 
$(X, d)$ 
is 
$R$-valued, 
then so is 
$(\yoexpsp{X}, \yoexpdis{d})$.
\end{enumerate}
\end{cor}

\begin{prop}\label{prop:18:powpow}
Let 
$R$ 
be  a range set, 
and 
$(X, d)$ 
be an  
$(R, \yoomega)$-haloed 
ultrametric space. 
Take 
$m\in \zz_{\ge 2}\sqcup \{\infty\}$, 
and 
$l\in (0, \infty]$. 
Then 
 the space 
$(\yoexpspml{X}{m}{l}, \yoexpdis{d})$ 
is 
$(R, \yoomega)$-haloed. 
\end{prop}
\begin{proof}
Take 
$A\in \yoexpspml{X}{m}{l}$ 
and 
$r\in R$.  
In the case of 
$m<\infty$, 
put 
$E=A$. 
In the case of
 $m=\infty$, 
take  a finite set 
$E$
with 
$\yoexpdis{d}(E, A)\le r$. 
Fix  
$x\in E$, 
and 
take a countable  
$r$-equidistant 
subset 
$H$ 
of 
$\yocball(x, r; d)$. 
We now construct 
an 
$r$-equidistant subset of 
$\yocball(A, r; \yoexpdis{d})$.

For each 
$q\in H$, 
put 
$F_{q}=(E\setminus \yocball(x, r; d))\cup \{q\}$. 
Note that if 
$l<r$, 
then 
$F_{q}=\{q\}$. 
Since 
$E$ 
is a finite set, 
each 
$F_{q}$ 
is closed in 
$X$. 
We also notice that 
$\card(F_{q})\le m$ 
and 
$\yodiam(F_{q})\le l$ 
for all 
$q\in H$.
Namely, 
for every 
$q\in H$ 
we have
$F_{q}\in \yoexpspml{X}{m}{l}$. 
By the fact that 
$H$ 
is 
$r$-equidistant,  
and 
by the definition of 
$F_{q}$, 
we obtain 
$\yoexpdis{d}(F_{q}, F_{q^{\prime}})=r$
for all distinct 
$q, q^{\prime}\in H$. 
Proposition 
\ref{prop:18:haus} 
shows that 
 $\yoexpdis{d}(F_{q}, E)\le r$, 
 and hence  
 $\yoexpdis{d}(F_{q}, A)\le r$
 for all 
 $q\in H$. 
Thus
$\{\, F_{q}\mid q\in H\}$ 
is a
countable 
$r$-equidistant 
subset of 
$\yocball(A, r; \yoexpdis{d})$. 
This proves the proposition. 
\end{proof}

\begin{rmk}
As a sophisticated version of 
Proposition
 \ref{prop:18:powpow},
we can prove the following statement:
Let 
$n\in \mm$, 
$m\in\zz_{\ge 2}\sqcup\{\infty\}$, 
and 
$R$ 
be a range set. 
If 
$(X, d)$ is 
$\youfin(R, n)$-injective, 
then 
the space 
$(\yoexpspml{X}{m}{\infty}, \yoexpdis{d})$ 
is 
$\youfin(R,  n+1)$-injective. 
Based on this statement
and 
Theorem 
\ref{thm:18:chara},  
the operations of 
 the hyperspace and symmetric products  can 
be considered as 
a 
Kat\v{e}tov functor in  the 
category of ultrametric spaces
(for  Kat\v{e}tov functors in a
category of ordinary metric spaces,  
see \cite{MR3669173}). 
\end{rmk}

We now clarify the 
metric structure of the 
hyperspace of the 
$R$-Urysohn universal 
ultrametric space in the case where 
$R$ 
is finite or countable. 
\begin{thm}\label{thm:18:countable}
If a range set 
$R$ 
is finite or  countable and 
$(X, d)$ 
is 
the 
$R$-Urysohn universal ultrametric space, then 
for every  
$m\in \zz_{\ge 2}\sqcup\{\infty\}$ 
and 
for every 
$l\in (0, \infty]$
the space
$(\yoexpspml{X}{m}{l}, \yoexpdis{d})$ 
is 
the 
$R$-Urysohn universal ultrametric space. 
In particular, 
it is isometric to 
$(X, d)$. 
\end{thm}
\begin{proof}
According to 
Corollary 
\ref{cor:18:val}, 
Proposition 
\ref{prop:18:powpow} 
and 
Theorem 
\ref{thm:18:chara}, 
the hyperspace 
$(\yoexpspml{X}{m}{l}, \yoexpdis{d})$
 is 
an 
$R$-valued
 $\youfin(R, \yoomega)$-injective 
 ultrametric space. 
On account of 
Lemma 
\ref{lem:18:folklore2}, 
we notice  that 
 $(\yoexpspml{X}{m}{l}, \yoexpdis{d})$ 
 is separable and complete. 
Therefore 
$(\yoexpspml{X}{m}{l}, \yoexpdis{d})$ 
is the 
$R$-Urysohn 
universal
ultrametric  space and 
isometric to 
$(X, d)$. 
\end{proof}

For the uncountable case, 
we need the following lemma.

\begin{lem}\label{lem:finalkey}
Let 
$R$ 
be an uncountable range set, 
and 
$(X, d)$ 
be the 
$R$-petaloid ultrametric space. 
Then for every 
$S\in \yoofam{R}$, 
and  for every finite subset 
 $A$
  of 
  $X$, 
there exists a subset 
$M$
 of 
 $\yopetal{X}{S}$ 
 such that:
\begin{enumerate}[label=\textup{(\arabic*)}]

\item\label{item:hdmax}
$\yoexpdis{d}(A, M)\le \max_{a\in A}d(a, \yopetal{X}{S})$;

\item\label{item:cardma}
$\card(M)\le \card(A)$;

\item\label{item:diamma}
$\yodiam(M)\le \yodiam(A)$. 
\end{enumerate}
\end{lem}
\begin{proof}
The lemma follows from 
Lemma
 \ref{lem:18:p6} 
 and 
\cite[Theorem 4.6]{artico1981some}
stating that every proximal subset of an 
(generalized) ultrametric space is a 
$1$-Lipschitz 
retract of the whole space. 
For the sake of self-containedness, 
we briefly explain a construction of $M$
based on the proof of 
\cite[Theorem 4.6]{artico1981some}. 
For each
 $a\in A$, 
we put 
$l_{a}=d(a, \yopetal{X}{S})$
 and 
we denote by 
$E[a]$
 the set of all 
 $y\in \yopetal{X}{S}$ 
 such that 
$d(a, y)=l_{a}$. 
For each set 
$E[a]$, 
take a point
 $p_{E[a]}\in E[a]$. 
Notice that if
 $E[a]=E[b]$,
  then 
  $p_{E[a]}=p_{E[b]}$. 
Lemma 
\ref{lem:18:p6} 
implies that 
$E_{a}\neq \emptyset$ 
for any 
$a\in A$. 
Define  
$M=\{\, p_{E[a]}\mid a\in A\, \}$. 
An important part of the construction of 
$M$ 
is 
that we choose 
$p_{E[a]}$  
so that 
it depends  on the set 
$E[a]$ 
rather than the point 
$a$.
Since 
$d(a, p_{E[a]})=d(a, \yopetal{X}{S})$ 
for all 
$a\in A$, 
the property 
\ref{item:hdmax} 
is true. 
The property 
\ref{item:cardma} 
follows 
from the definition of 
$M$. 
Next we verify 
\ref{item:diamma}. 
It is sufficient to prove 
that 
$d(p_{E[a]}, p_{E[b]})\le d(a, b)$ 
for all 
$a, b\in A$. 
If 
$l_{a}\lor l_{b}\le d(a, b)$, 
then 
the strong triangle inequality implies 
$d(p_{E[a]}, p_{E[b]})\le d(a, b)$. 
If 
$d(a, b)<l_{a}\lor l_{a}$, 
then we can prove 
$l_{a}=l_{b}$
and 
 $E[a]=E[b]$. 
 Thus
$p_{E[a]}=p_{E[b]}$. 
In this situation, 
of course, 
we have 
$d(p_{E[a]}, p_{E[b]})\le d(a, b)$. 
\end{proof}

For a metric space 
$(X, d)$ 
and for  
$r\in (0, \infty)$, 
a subset of 
$A$ 
is called an 
\yoemph{$r$-net}
if for all distinct 
$x, y\in A$ 
we have 
$r< d(x, y)$. 
An  
$r$-net is 
\yoemph{maximal} if 
it is maximal with respect to  inclusion 
$\yosub$. 
Let us prove 
Theorem 
\ref{thm:18:hyper}. 
\begin{proof}[Proof of Theorem \ref{thm:18:hyper}]
The case where 
$R$ 
is finite or  countable is proven in 
Theorem 
 \ref{thm:18:countable}. 
We now treat the case where 
$R$ 
is uncountable. 
Remark that  in this case, 
$(X, d)$ 
is the 
$R$-petaloid space
(see Definition \ref{df:18:origin}). 
Corollary 
\ref{cor:18:val}
 implies that 
$(\yoexpspml{X}{m}{l}, \yoexpdis{d})$
 is 
an 
$R$-valued ultrametric space. 
We define a petal on 
$(\yoexpspml{X}{m}{l}, \yoexpdis{d})$ 
as follows. 
For every   
$S\in \yoofam{R}$, 
we define 
\[
\yopetal{\yoexpspml{X}{m}{l}}{S}
=
\{\, E\in \yoexpspml{X}{m}{l}\mid E\yosub \yopetal{X}{S}\, \}.
\] 
By this definition, 
the property 
\ref{item:pr:cap} 
is satisfied.

Since for every 
$S\in \yoofam{R}$
the space 
$(\yopetal{X}{S}, d)$
 is the 
$S$-Urysohn 
universal ultrametric space
(the property 
\ref{item:pr:sep}  
for 
$(X, d)$), 
according to  
Theorem
 \ref{thm:18:countable}, 
the space 
$(\yopetal{\yoexpspml{X}{m}{l}}{S}, \yoexpdis{d})$
 is 
the $S$-Urysohn universal ultrametric space. 
Namely, 
the property 
\ref{item:pr:sep} 
is true for 
$({\yoexpspml{X}{m}{l}}, \yoexpdis{d})$. 
Due to 
\cite[Proposition 2.8]{Ishiki2023Ury}, 
the 
property 
\ref{item:pr:cup} 
is satisfied.

We now  show that 
the property 
\ref{item:pr:distance} 
is valid 
for 
$(\yoexpspml{X}{m}{l}, \yoexpdis{d})$. 
Take 
$S, T\in \yoofam{R}$, 
 and take 
$E\in \yopetal{\yoexpspml{X}{m}{l}}{T}$. 
Let 
$L$ 
stand for 
the distance  between 
the point 
$E$ 
and 
the set 
$\yopetal{\yoexpspml{X}{m}{l}}{S}$. 
We shall 
prove  that 
$L$ 
belongs to 
$\{0\}\cup (T\setminus S)$. 
We may assume that 
$E\not\in  \yopetal{\yoexpspml{X}{m}{l}}{S}$. 
Then 
$E\yosub \yopetal{X}{T}$, 
$E\not\yosub\yopetal{X}{S}$, 
and 
$L>0$.
We put 
$H=\{\, d(y, \yopetal{X}{S})\mid y\in E\, \}$. 
Using  the property 
\ref{item:pr:distance} 
for 
$(X, d)$, 
we see that  
$H\yosub \{0\}\cup (T\setminus S)$. 
Put 
$h=\max H$ 
and take 
$z\in E$ 
such that 
$d(z, \yopetal{X}{S})=h$.
The existence of 
$h$ 
is guaranteed by 
the fact that 
 $H$  
 is a subset of 
 the  tenuous set 
 $T$.
 Note that 
 $h\in T\setminus S$
 and 
 $h>0$.

From now on, we
shall  confirm that
 $L=h$. 
 To prove 
 $h\le L$, 
for the sake of contradiction, 
suppose that there exists 
$F\in \yopetal{\yoexpspml{X}{m}{l}}{S}$
with 
$\yoexpdis{d}(E, F)<h$. 
Put 
$u=\yoexpdis{d}(E, F)$. 
Then 
Lemma 
\ref{lem:18:folklore1} 
and 
$u<h$ 
imply that 
there exists 
$f\in F$
such that 
$d(z, f)<h$. 
This is a
contradiction to 
$h=\inf_{x\in \yopetal{X}{S}}d(z, x)$. 
Thus, 
 for every
$F\in \yopetal{\yoexpspml{X}{m}{l}}{S}$, 
we have 
$h\le \yoexpdis{d}(E, F)$. 
Hence 
 $h\le L$. 

 Next we show 
 $L=h$. 
Take a maximal 
 finite  
$h$-net 
$Q$ 
of 
$E$ 
such that 
$Q\yosub E$. 
Notice that 
$Q\in \yopetal{\yoexpspml{X}{m}{l}}{S}$. 
Lemma
 \ref{lem:finalkey} enables 
us to take a subset 
$V$ 
of 
$\yopetal{X}{S}$ 
such that 
$\yoexpdis{d}(Q, V)\le h$, 
$\card(V)\le \card(Q)$, 
and 
$\yodiam(V)\le \yodiam(Q)$. 
Since 
$Q\in \yopetal{\yoexpspml{X}{m}{l}}{S}$, 
we notice that 
$V\in \yoexpspml{X}{m}{l}$, 
and hence 
$V\in \yopetal{\yoexpspml{X}{m}{l}}{S}$. 
By the fact that  
$Q$ 
is a maximal 
$h$-net of 
$E$, 
we obtain 
$\yoexpdis{d}(E, Q)\le h$. 
Since 
$\yoexpdis{d}$ 
is an ultrametric 
(see the statement 
\ref{item:18:ult} 
in Corollary 
\ref{cor:18:val}), 
combining 
$\yoexpdis{d}(Q, V)\le h$ 
and 
$\yoexpdis{d}(E, Q)\le h$, 
the strong triangle inequality 
 shows  that 
$\yoexpdis{d}(E, V)\le h$.  
Hence 
$L\le h$. 
By 
$V\in \yopetal{\yoexpspml{X}{m}{l}}{S}$ 
and 
$h\le L$, 
we 
conclude that
 $L=h$. 
Therefore 
$L\in T\setminus S$. 
This finishes the proof of 
Theorem 
\ref{thm:18:hyper}. 
\end{proof}

\begin{rmk}
Theorem
 \ref{thm:18:hyper} 
can be
 considered as a
  non-Archimedean  analogue of 
the 
statement 
 that the Gromov--Hausdorff space  is 
isometric to the quotient metric space of 
hyperspace of the Urysohn universal metric space
(see 
\cite[Exercise (b) in the page 83]{Greenbook1999}, 
\cite[Theorem 3.4]{MR4052556}, 
and 
 \cite{MR4282005}). 
\end{rmk}

\begin{rmk}
Let 
$R$ 
be 
a range set, and 
$(X, d)$ 
be the 
$R$-Urysohn universal ultrametric space. 
There does not seem to exist 
a natural isometry between 
$(X, d)$ 
and 
$(\yoexpsp{X}, \yoexpdis{d})$. 
In the case of 
$R=\{0, 1\}$, 
the space 
$(X, d)$ 
is the countable discrete space and 
all bijections between 
$X$ 
and 
$\yoexpsp{X}$ 
become  isometries between them. 
\end{rmk}

\begin{rmk}
Theorem 
\ref{thm:18:hyper} 
implies  that 
there is an (non-trivial) ultrametric space 
$(X, d)$ 
such that 
$(\yoexpsp{X}, \yoexpdis{d})$ 
is isometric to 
$(X, d)$. 
In contrast, 
in \cite{MR2054513},
 it is shown that for any bounded metric space
$(X, d)$ 
with more than one point,  the space 
$(\mathrm{CL}(X), \yoexpdis{d})$ 
is 
not isometric to 
$(X, d)$, 
where 
$\mathrm{CL}(X)$ 
stands for the 
set of all non-empty  closed bounded subsets of 
$(X, d)$. 
\end{rmk}

\section{Urysohn universal ultrametric  spaces as subsets}\label{sec:existence}
This section is devoted to the proof of  
Theorem 
\ref{thm:18:ram}. 
Throughout this section, 
the symbols
 $R$, 
 $(X, d)$ 
 and 
 $K$ 
 are 
the same objects
as
 in  the statement in Theorem \ref{thm:18:ram}. 
Namely,
$R$ 
is an uncountable range set,  
$(X, d)$ 
is 
a
complete 
$\youfin(R, \yoomega)$-injective 
ultrametric space, 
and 
$K$
is 
a compact  subset  of 
$X$ 
satisfying  
$d(K^{2})\yosub R$.

Before proving the theorem, 
we explain the plan of the proof. 
To obtain  a subspace 
$F$ 
stated in Theorem \ref{thm:18:ram}, 
we shall construct a petal structure
consisting of subsets of
 $X$. 
For this purpose, 
we use the idea of a family tree, 
or 
a  genealogical tree. 
This strategy  is inspired by  the fact that 
every ultrametric space can be seen as 
the end space 
 of a tree.
We fix a point 
$\yoprotobio\in X$ 
and 
we regard 
$\yoprotobio$ 
 as an original ancestor, 
 or a
  protobiont.
For every  
$S\in \yoofam{R}$, 
we also define  
$S$-heirs 
of
 $\yoprotobio$
  by  sequences beginning from 
  $\yoprotobio$, 
which also can be regarded as descendants of
 $\yoprotobio$ 
 passing through 
 $S$.
Considering the completion 
$\yopsp{S}$ 
of 
 the set 
 $\yoxsp{S}$ 
 of all 
 $S$-heirs 
 of 
 $\yoprotobio$, 
 we can observe that 
the family 
$\{\yopsp{S}\}_{S\in \yoofam{R}}$ 
becomes a petal structure contained in 
$X$. 
Then  
the space 
$F=\bigcup_{S\in \yoofam{R}}\yopsp{S}$ 
is 
as desired.

We begin with a concept that is  
 a consequence of Theorem 
 \ref{thm:18:chara}. 
\begin{df}\label{df:18:seeds}
We say that a family 
 $\{\yoseed{a}{r}\}_{a\in X, r\in R\setminus \{0\}}$ 
 of 
 subsets of 
 $X$ 
 is 
 an
 \yoemph{$R$-seed of $X$}
  if 
 the following conditions are satisfied
 for every 
 $a\in X$, 
 and 
 for every 
 $r\in R\setminus \{0\}$:  
\begin{enumerate}
\item 
we have 
$a\in \yoseed{a}{r}$;
\item 
the set 
$\yoseed{a}{r}$ 
is 
$r$-equidistant;
\item 
we have 
$\card(\yoseed{a}{r})=\aleph_{0}$. 
\end{enumerate}
\end{df}
Notice that since 
$(X, d)$ 
is 
$(R, \yoomega)$-haloed
  (see Theorem 
  \ref{thm:18:chara}), 
there
  exists an 
  $R$-seed, 
  and 
  notice 
  that 
each 
$\yoseed{a}{r}$ 
is not necessarily maximal with respect to 
inclusion 
$\yosub$. 
\begin{df}\label{df:18:inheritance}
Fix an 
$R$-seed 
$\{\yoseed{a}{r}\}_{a\in X, r\in R\setminus \{0\}}$ 
and 
a point 
$\yoprotobio\in X$. 
For 
$S\in \yoofam{R}$, 
we say that  a point 
$x\in X$ 
is 
an 
\yoemph{$S$-heir 
of  
$\yoprotobio$} 
if 
there exists sequences
$\{v_{i}\}_{i=0}^{m}$ 
in 
$X$ 
and 
$\{r_{i}\}_{i=0}^{m-1}$ 
in 
$S$ 
such that: 
\begin{enumerate}[label=\textup{(S\arabic*)}]

\item\label{item:18seed0}
$v_{0}=\yoprotobio$;

\item\label{item:18seedend}
$v_{m}=x$; 

\item\label{item:18seedeq}
$v_{i+1}\in \yoseed{v_{i}}{r_{i}}$
 for all 
 $i\in \{0, \dots, m-1\}$; 

\item\label{item:18seedneq}
$v_{i}\neq v_{i+1}$ 
for 
all
 $i\in \{0, \dots, m-1\}$; 

\item\label{item:18seeddec}
if 
$2\le m$, 
then
$r_{i+1}<r_{i}$ for all 
$i\in \{0, \dots, m-2\}$. 
\end{enumerate}
In this case, 
 the pair of 
 $\{v_{i}\}_{i=0}^{m}$ 
 and 
 $\{r\}_{i=0}^{m-1}$ 
 is called 
an  
\yoemph{$S$-inheritance 
of 
$x$ 
from 
$\yoprotobio$ 
with length 
$m$}. 
If 
$m=0$, 
then we consider 
$\{r_{i}\}_{i=0}^{m-1}$ 
as 
 the empty sequence. 
\end{df}

\begin{df}\label{df:18:xpsp}
For an  
$R$-seed 
$\{\yoseed{a}{r}\}_{a\in X, r\in R\setminus \{0\}}$, 
and a point  
$\yoprotobio\in X$, 
we denote by 
$\yoxsp{S}$
the set of  all 
$S$-heirs 
of 
$\yoprotobio$. 
We also denote by 
$\yopsp{S}$
the closure of 
$\yoxsp{S}$ 
in 
$X$. 
\end{df}
Notice that since 
$(X, d)$ 
is complete, 
the
set 
$\yopsp{S}$ 
is isometric to 
the 
completion of 
$\yoxsp{S}$
for all 
$S\in \yoofam{R}$.

\begin{rmk}
The point 
$\yoprotobio$
 is 
the unique 
$S$-heir 
from 
$\yoprotobio$ 
that has 
an 
$S$-inheritance 
with 
length 
$0$ for all 
$S\in \yoofam{R}$. 
Thus, 
we have 
$\yoxsp{\{0\}}=\{\yoprotobio\}$
and 
$\yopsp{\{0\}}=\{\yoprotobio\}$. 
\end{rmk}

In the next proposition, 
we observe that 
inheritances of two points 
determine a distance between these two points. 
\begin{prop}\label{prop:18:eqeq}
Let 
$S, T\in \yoofam{R}$, 
and 
$x$ 
and 
$y$ 
be 
$S$-heir 
and 
$T$-heir 
of 
$\yoprotobio$,
respectively.  
Let 
$(\{v_{i}\}_{i=0}^{m}, \{r_{i}\}_{i=0}^{m-1})$
and 
$(\{w_{i}\}_{i=0}^{n}, \{l_{i}\}_{i=0}^{n-1})$
be 
$S$-inheritance 
and 
$T$-inheritance 
of 
$x$ 
and 
$y$, 
respectively. 
In this situation, 
the next three statements are true:
\begin{enumerate}[label=\textup{(\arabic*)}]

\item\label{item:42:k1}
If 
$k\in \mm$ 
satisfies 
 that 
$k\le \min\{m, n\}$ 
and 
$v_{i}=w_{i}$ 
for all 
$i\le k$, 
then 
$r_{i}=l_{i}$ 
for all 
$i< k$.

\item\label{item:42:k2}
If 
$k\in \mm$ 
satisfies that:
\begin{enumerate}[label=\textup{(\alph*)}]

\item 
$v_{i}=w_{i}$ 
for all 
$i\le k$;

\item
$v_{k+1}\neq w_{k+1}$, 
\end{enumerate}
then 
we have 
$\max\{r_{k}, l_{k}\}=d(x, y)$.
\item\label{item:42:k3}
If 
$m<n$ 
and  
$v_{i}=w_{i}$ 
for all 
$i\le m$, 
then we have 
$d(x, y)=l_{m}$. 
\end{enumerate}
\end{prop}
\begin{proof}
Since 
$r_{i}=d(v_{i}, v_{i+1})$ 
and 
$l_{i}=d(w_{i}, w_{i+1})$
(see the condition 
\ref{item:18seedeq}
in Definition 
\ref{df:18:inheritance}), 
under the assumption of 
\ref{item:42:k1} 
we see that 
$r_{i}=l_{i}$ 
for all 
$i<k$. 
This finishes the proof of 
\ref{item:42:k1}.

We next prove \ref{item:42:k2}. 
Put 
$z=v_{k}$. 
Then 
$z=w_{k}$ 
by the assumption. 
By the definition of heirs,  
the strong triangle inequality, 
and by the conditions \ref{item:18seedend} 
and  
\ref{item:18seeddec} 
in 
Definition 
\ref{df:18:inheritance}, 
we have 
\begin{align*}
d(v_{k+1}, x)&\le d(v_{k}, v_{k+1})\lor \dots \lor d(v_{m-1}, v_{m})\\
&=\max\{\, r_{i}\mid i\in \{k+1, \dots, m-1\}\, \}<r_{k},
\end{align*}
and hence 
$d(v_{k+1}, x)<r_{k}$. 
Similarly, 
we also have 
$d(w_{k+1}, y)<l_{k}$. 
Since 
$d(v_{k}, v_{k+1})=r_{k}$
and 
$d(w_{k}, w_{k+1})=l_{k}$, 
the statement 
\ref{item:18:ball0} 
in 
Lemma 
\ref{lem:ultraopcl} 
yields
$d(z, x)=d(v_{k}, x)=r_{k}$
 and 
$d(z, y)=d(w_{k}, y)=l_{k}$. 
If 
$r_{k}\neq l_{k}$, 
then 
we obtain 
$d(x, y)=\max\{r_{k}, l_{k}\}$
using the statement 
\ref{item:18:ball0} 
in 
Lemma 
\ref{lem:ultraopcl} 
again. 
If 
$r_{k}=l_{k}$, 
then 
$v_{k+1}, w_{k+1}\in \yoseed{z}{r_{k}}$
(see 
\ref{item:18seedeq} in 
Definition 
\ref{df:18:inheritance}). 
Thus, 
due to 
$v_{k+1}\neq w_{k+1}$, 
we have  
$d(v_{k+1}, w_{k+1})=r_{k}(=l_{k})$. 
Using the statement 
\ref{item:18:ball0} 
in 
Lemma 
\ref{lem:ultraopcl} 
once  again, 
and by 
$d(v_{k+1}, x)<r_{k}$ 
and 
$d(w_{k+1}, y)<r_{k}$, 
we obtain 
$d(x, y)=r_{k}(=l_{k})$.
This proves the statement 
\ref{item:42:k2}.

Under the assumption of 
\ref{item:42:k3}, 
we obtain    
$d(w_{m}, w_{m+1})=l_{m}$
 and 
$d(w_{m+1}, y)<l_{m}$.
The condition 
\ref{item:18seedend} 
states that 
$w_{m}=x$. 
Thus, 
the statement 
\ref{item:18:ball0} 
in 
Lemma 
\ref{lem:ultraopcl}
yields
$d(x, y)=d(w_{m}, y)=l_{m}$. 
The proof is finished. 
\end{proof}

\begin{rmk}\label{rmk:18:unique}
Let 
$S\in \yoofam{R}$
 and 
fix 
$x\in \yoxsp{S}$. 
Proposition 
\ref{prop:18:eqeq} 
implies 
the uniqueness of 
an  
$S$-inheritance 
$(\{v_{i}\}_{i=0}^{m}, \{r_{i}\}_{i=0}^{m-1})$
of 
$x$
 from 
 $\yoprotobio$. 
\end{rmk}

\begin{lem}\label{lem:18:Acountable}
For every 
$S\in \yoofam{R}$, 
the set 
 $\yoxsp{S}$ 
 is countable.
\end{lem}
\begin{proof}
For each 
$n\in \mm$, 
let 
$A_{n}$ 
be 
the set of  
$S$-heirs from 
$\yoprotobio$ 
with  length
 $n$. 
Notice that
 $A_{0}=\{\yoprotobio\}$. 
Then 
$\yoxsp{S}=\bigcup_{n\in \mm}A_{n}$. 
Remark  that 
$A_{n+1}$ 
can be 
regarded
 as 
 a
subset of the set 
$\bigcup_{a\in A_{n}, r\in S}\{a\}\times \yoseed{a}{r}$ 
by 
identifying  a point
 $x\in A_{n+1}$
  with 
its  
$S$-inheritance. 
Since 
$S$ 
and each
 $\yoseed{a}{r}$ 
 are countable, 
by induction, 
we see that each 
$A_{n}$ 
is countable. 
Thus 
$\yoxsp{S}$
is countable. 
\end{proof}

We now prove that each 
$(\yopsp{S}, d)$
 is the 
 $S$-Urysohn
  universal ultrametric space. 
\begin{lem}\label{lem:18:ESiso}
For every 
 $S\in \yoofam{R}$, 
the space 
$(\yopsp{S}, d)$ 
is 
the 
$S$-Urysohn 
universal ultrametric space. 
\end{lem}
\begin{proof}
Since 
$\yoxsp{S}$ 
is countable
(see Lemma
 \ref{lem:18:Acountable}), 
the space 
$\yopsp{S}$ 
is separable. 
From 
\ref{item:42:k2} 
and 
\ref{item:42:k3} 
in 
Proposition 
\ref{prop:18:eqeq},
it follows that  
 $(\yopsp{S}, d)$ 
 is 
$S$-valued. 
We now  
verify that 
$(\yopsp{S}, d)$ 
is 
$(S, \yoomega)$-haloed.
It suffices to show that for 
every
$a\in \yoxsp{S}$ 
and for 
every
$r\in S\setminus \{0\}$, 
there
 exists an infinite 
$r$-equidistant 
subset of 
$\yocball(a, r)$. 
Let 
$(\{v_{i}\}_{i=0}^{m}, \{r_{i}\}_{i=0}^{m-1})$
be an 
$S$-inheritance 
of 
$a$. 
We divide the proof into three cases.

\yocase{1}{$r_{0}<r$}
In this case, 
the set 
$\yoseed{\yoprotobio}{r}$ is 
contained in 
$\yopsp{S}$. 
Indeed, for every 
$q\in \yoseed{\yoprotobio}{r}\setminus \{\yoprotobio\}$, 
define 
$\{w_{i}\}_{i=0}^{1}$ and $\{l_{i}\}_{i=0}^{0}$
by 
$w_{0}=\yoprotobio$, 
$w_{1}=q$, 
and 
$l_{0}=r$. 
Then 
$(\{w_{i}\}_{i=0}^{1}, \{l_{i}\}_{i=0}^{0})$ is 
an 
$S$-inheritance 
of 
$q$ 
from 
$\yoprotobio$. 
Hence 
$\yoseed{\yoprotobio}{r}\yosub \yopsp{S}$. 
By the statement 
\ref{item:42:k2} in 
Proposition 
\ref{prop:18:eqeq}, 
we obtain 
$\yoseed{\yoprotobio}{r}\yosub \yocball(a, r)$. 
Thus 
the set 
$\yoseed{\yoprotobio}{r}$
 is 
as desired.

\yocase{2}{There exists $k$ such that $r_{k+1}<r<r_{k}$}
To prove 
$\yoseed{v_{k}}{r}\yosub \yopsp{S}$, 
for each 
$q\in \yoseed{v_{k}}{r}\setminus \{v_{k}\}$, 
we define 
$(\{w_{i}\}_{i=0}^{k+1}, \{l_{i}\}_{i=0}^{k})$
by 
\[
w_{i}=
\begin{cases}
v_{i} & \text{if $i\le k$;}\\
q & \text{if $i=k+1$}. 
\end{cases}
\]
and 
\[
l_{i}=
\begin{cases}
r_{i} & \text{if $i\le k$;}\\
r & \text{if $i=k+1$}. 
\end{cases}
\]
Then 
$(\{w_{i}\}_{i=0}^{k+1}, \{l_{i}\}_{i=0}^{k})$ 
is 
an 
$S$-inheritance 
of 
$q$ 
from 
$\yoprotobio$, 
and hence 
$\yoseed{v_{k}}{r}\yosub \yopsp{S}$. 
By the statement 
\ref{item:42:k2}
 in 
Proposition 
\ref{prop:18:eqeq}, 
we obtain 
$\yoseed{v_{k}}{r}\yosub \yocball(a, r)$. 
Thus 
$\yoseed{v_{k}}{r}$ is 
an
infinite 
$r$-equidistant 
subset as
required. 

\yocase{3}{$r<r_{i}$ for all $i\in \{0, \dots, m-1\}$}
In a similar way to the cases explained  above, 
we observe that 
 $\yoseed{a}{r}\yosub \yocball(a, r)$.

Thus, we conclude that 
$(\yopsp{S}, d)$
is 
$(S, \yoomega)$-haloed. 
Due to Theorem 
\ref{thm:18:chara}, 
the space  
$(\yopsp{S}, d)$
 is 
$\youfin(S, \yoomega)$-injective, 
and hence 
it is the 
$S$-Urysohn universal ultrametric space. 
\end{proof}

\begin{lem}\label{lem:18:banana}
Let 
$S, T\in \yoofam{R}$. 
Let 
$x$ 
be a 
$T$-heir of $\yoprotobio$, 
and 
$(\{v_{i}\}_{i=0}^{m}, \{r_{i}\}_{i=0}^{m-1})$ 
be 
a 
$T$-inheritance of 
$x$. 
If 
$x\not \in\yopsp{S}$, 
then there exists 
$i\in \{0, \dots, m\}$ 
such that
 $r_{i}\in T\setminus  S$. 
\end{lem}
\begin{proof}
If all 
$r_{i}$
 were in 
$S$, 
then 
we would have
 $x\in \yopsp{S}$. 
\end{proof}

The following lemma is a 
preparation for 
the proof of 
\ref{item:pr:distance}. 
\begin{lem}\label{lem:18:TS}
Let 
$S, T\in \yoofam{R}$, 
and 
$x\in \yopsp{T}$. 
Then 
the value 
$d(x, \yopsp{S})$
 belongs to 
$\{0\}\cup (T\setminus S)$. 
\end{lem}
\begin{proof}
We may assume that 
$x\not\in 
\yopsp{S}$. 
Put 
$L=d(x, \yopsp{S})$. 
Then 
$L>0$. 
Take a point 
$z\in \yoxsp{T}$ 
sufficiently close
to 
$x$ 
so that
 $d(z, \yopsp{S})=L$ 
 and 
$z\in \yopsp{T}\setminus \yopsp{S}$. 
Let 
$(\{v_{i}\}_{i=0}^{m}, \{r_{i}\}_{i=0}^{m-1})$ 
be a 
$T$-inheritance 
of 
$z$ 
from 
$\yoprotobio$. 
Since 
$L>0$ 
and 
$\yoprotobio\in \yopsp{S}$,
 we notice that 
$z\neq \yoprotobio$. 
Thus 
$0<m$. 
Let 
$k$ 
be the minimal number such that 
$r_{k}\in T\setminus S$. 
The existence of 
$k$ 
is guaranteed by 
Lemma 
\ref{lem:18:banana}. 
From 
\ref{item:42:k2}
 in Proposition 
 \ref{prop:18:eqeq}, 
it follows that 
every  
$y\in \yoxsp{S}$ 
satisfies 
$d(z, y)\ge r_{k}$. 
If 
$k=0$, 
we define  
$q=\yoprotobio$; 
otherwise 
$q=v_{k-1}$. 
Then 
$q\in \yopsp{S}$ 
and 
the statement 
\ref{item:42:k3} 
in 
Proposition 
\ref{prop:18:eqeq} 
implies that 
$d(z, q)=r_{k}$. 
As a result, 
we have   
$L=r_{k}$. 
Since 
$r_{k}\in T\setminus S$, 
we conclude that 
$L\in T\setminus S$. 
\end{proof}

We now give the proof of Theorem \ref{thm:18:ram}. 
\begin{proof}[Proof of Theorem \ref{thm:18:ram}]
It suffices to show the latter part 
of Theorem 
\ref{thm:18:ram}. 
Choose a point  
$\yoprotobio$ 
such that 
$\yoprotobio \in K$ and 
choose an 
$R$-seed 
$\{\yoseed{a}{r}\}_{a\in X, r\in R\setminus \{0\}}$ 
so  that the following condition is true:
\begin{enumerate}[label=\textup{(M)}]

\item\label{item:18:kcondition}
If 
$a\in K$, 
then for every 
$r\in R\setminus \{0\}$
the set 
$\yoseed{a}{r}\cap K$ 
is a 
maximal 
$r$-equidistant set  of 
$K$. 
\end{enumerate}
This is possible since 
$K$ 
is compact and 
$d(K^{2})\yosub R$. 
Notice that the following statements are true: 
\begin{enumerate}[label=\textup{(\alph*)}]
\item\label{item:18:aaaa}
For every 
$a\in K$ 
and every 
$r\in R$, 
the set 
$\yoseed{a}{r}\cap K$ 
is 
contained in 
$\yocball(a, r)\cap K$; 

\item\label{item:18:bbbb}
For every 
$a\in K$,  
for every 
$r\in R$,  
and 
for every 
$p\in \yocball(a, r)\cap K$, 
there exists 
$w\in \yoseed{a}{r}\cap K$ 
such that 
$d(p, w)<r$. 
\end{enumerate}
Define 
$F=\bigcup_{S\in \yoofam{R}}\yopsp{S}$
and define a petal of 
$F$ 
by 
$\yopetal{F}{S}=\yopsp{S}$. 
By this definition, 
we can 
 confirm that 
the properties 
\ref{item:pr:cup}, 
and 
\ref{item:pr:cap} 
are 
satisfied. 
Using 
Lemma 
\ref{lem:18:ESiso}, 
the property 
\ref{item:pr:sep} 
is 
true. 
Lemma 
\ref{lem:18:TS} 
proves 
the property 
\ref{item:pr:distance}. 
Therefore 
$(F, d)$ 
is 
the 
$R$-petaloid ultrametric space.

Next 
we  prove 
$K\yosub F$. 
Take
an arbitrary point  
$x\in K$. 
If
 $x\in \yoxsp{S}$
for some 
$S\in \yoofam{R}$, 
then 
$x\in F$. 
Thus we may assume that 
$x\not \in \yoxsp{S}$
for any 
$S\in \yoofam{R}$. 
Put 
$T=d(K^{2})$. 
In this setting, 
by recursion, 
due to the assumption 
\ref{item:18:kcondition}
(or the statements
 \ref{item:18:aaaa} 
 and 
 \ref{item:18:bbbb}), 
we can define sequences 
$\{v_{i}\}_{i=0}^{\infty}$ 
in 
$X$ 
and 
$\{r_{i}\}_{i=0}^{\infty}$ 
in 
$T\setminus \{0\}$ 
such that:
\begin{enumerate}

\item 
$v_{0}=\yoprotobio$;

\item 
$v_{i+1}\in \yoseed{v_{i}}{r_{i}}\cap K$
for all 
$i\in \mm$;

\item 
$v_{i}\neq v_{i+1}$ 
for all 
$i\in \mm$;

\item 
we have 
$r_{i+1}<r_{i}$ 
for all 
$i\in \mm$;

\item 
$d(x, v_{i})<r_{i}$
for all 
$i\in \mm$. 
\end{enumerate}
Since 
$T$ 
is tenuous 
(see 
\cite[Corollary 2.28]{Ishiki2023const}), 
we have 
$r_{i}\to 0$ 
as 
$i\to \infty$,  
and hence 
$v_{i}\to x$ 
as 
$i\to \infty$. 
By the definitions, 
for each 
$m\in \mm$, 
the pair 
$(\{v_{i}\}_{i=0}^{m}, \{r_{i}\}_{i=0}^{m-1})$
becomes a
$T$-inheritance 
of 
$v_{m}$ from $\yoprotobio$. 
Subsequently, 
each 
$v_{m}$ 
belongs to 
$\yoxsp{T}$. 
Therefore, 
from the completeness of 
$\yopsp{T}$ 
and 
$v_{i}\to x$ 
as 
$i\to \infty$, 
it follows that 
$x\in \yopsp{T}$. 
Hence 
$K\yosub F$. 
We  complete the proof of 
Theorem 
\ref{thm:18:ram}. 
\end{proof}

\begin{rmk}
For
 $S\in \yoofam{R}$, 
an element 
$x\in \yoxsp{S}$ 
(or 
an 
$S$-inheritance of 
$x$ 
from 
$\yoprotobio$) 
can be translated into 
 a member in 
  $\yomapsco{R}{\yoomega}$
 (see Example 
 \ref{exam:gsp})
  with 
 finite support as follows. 
For every 
$a\in X$ 
and 
every
$r\in R\setminus \{0\}$, 
 fix an  enumeration  
  $\yoseed{a}{r}=\{\, q(a, r, k)\mid k\in\mm\, \}$
  and let 
  $(\{v_{i}\}_{i=0}^{m}, \{r_{i}\}_{i=0}^{m-1})$
  be an 
  $S$-inheritance 
  of 
  $x$ 
  from 
  $\yoprotobio$. 
Then the point 
$x$ 
is corresponding to 
the function 
$f_{x}$ 
in 
$\yomapsco{S}{\yoomega}$ 
defined by 
\[
f_{x}(s)=
\begin{cases}
k & \text{if $s=r_{i}$ for some $i$ and 
$v_{i+1}=q(v_{i}, r_{i} , k)$};\\
0 & \text{otherwise.} 
\end{cases}
\]
\end{rmk}
Note that 
the point 
$\yoprotobio$
 is corresponding to 
 the zero function
 in 
 $\yomapsco{R}{\yoomega}$. 
 
 \begin{rmk}
 Theorem \ref{thm:18:ram} is stil valid even if 
 $K$ is separable and 
 $d(K^{2})$ is tenuous. 
 \end{rmk}

\section{Questions}\label{sec:ques}

In 
\cite{MR611633}, 
for an index set 
$T$, 
Bors\'{\i}k 
and 
Dobo\v{s} 
introduced 
the notion of a metric preserving function 
$F\colon [0, \infty)^{T}\to [0, \infty)$, 
and the product metric 
of a family 
$\{(X_{t}, d_{t})\}_{t\in T}$
associated with 
$F$
(for the precise definition, 
see
 \cite{MR611633}). 
This is a generalization of 
the 
$\ell^{p}$-product 
metric for 
$p\in [1, \infty]$. 
In the case of two variables 
($\card(T)=2$), 
we denote by 
$d_{1}\times_{F}d_{2}$
the their product metric 
associated with 
$F$.

As a counter part of  of
 Theorem 
 \ref{thm:18:product}, 
we prove 
Theorem \ref{thm:18:nonury} stating that 
$(\yourysp\times \yourysp, \yourydis\times_{F}\yourydis)$ 
is not 
$\yofin$-injective. 
By generalizing the product metric, 
we raise the next question. 
\begin{ques}
Is there  
a  
metric preserving function 
$F\colon [0, \infty)^{2}\to [0, \infty)$
 of two variables such that 
$(\yourysp\times \yourysp, \yourydis\times_{F}\yourydis)$ 
is 
$\yofin$-injective?
\end{ques}

As an Archimedean analogue of  
Theorem 
\ref{thm:18:hyper}, 
we ask the following question. 
\begin{ques}\label{ques:hyp}
Is 
$(\yoexpsp{\yourysp}, \yoexpdis{\yourydis})$ 
isometric to 
$(\yourysp, \yourydis)$?
\end{ques}
The author suspects that 
Question 
\ref{ques:hyp} 
is negative.

For a suitable range  subset 
$R\yosub [0, \infty)$, 
we can construct 
a (unique) separable complete  metric space 
$(\yourysp_{R}, \yourydis_{R})$ 
associated with 
$R$
 injective for all finite  metric spaces 
 whose distances belong to 
$R$ 
(see, 
for instance, 
\cite{MR3004464}). 
Note that 
$\yourydis_{R}$ 
takes only values in 
$R$. 
In the case of 
$R=[0, \infty)$, 
the space 
$(\yourysp_{R}, \yourydis_{R})$
is nothing but 
$(\yourysp, \yourydis)$. 
The 
\yoemph{random graph}
(or the 
\yoemph{Rado graph}) 
$G$ is 
a (unique) countable graph with the 
following property. 
\begin{enumerate}[label=\textup{(E)}]
\item\label{item:18:extr}
For all finite vertices 
$u_{1}$, 
\dots, 
$u_{m}$, 
and 
$v_{1}$, 
\dots, 
$v_{n}$ 
of  
$G$, 
there exists a vertex 
$p$ 
of
 $G$ 
 which is 
adjunct
 $u_{1}$, 
 \dots, 
 $u_{m}$, 
  and 
not adjunct to 
$v_{1}$, \dots, $v_{n}$. 
\end{enumerate}
For the  definition using probability and   characterizations, see, for example, 
 \cite{MR1425227}. 
The random graph  is studied in
graph theory, and  model theory. 
Even though
 $G$ 
 is defined 
purely in terms of  
graph theory, 
the graph 
$G$ 
can be 
identified with 
the metric space 
$(\yourysp_{\{0, 1, 2\}}, \yourydis_{\{0, 1, 2\}})$
 by 
declaring that 
$x, y\in \yourysp_{\{0, 1, 2\}}$ 
are adjunct if and only if 
$\yourydis_{\{0, 1, 2\}}(x, y)=1$
(see 
\cite[Exercise 5]{MR2435148}),  
and 
the property 
\ref{item:18:extr} 
is corresponding to 
the injectivity for the class of all finite metric spaces whose 
distances take values in 
$\{0, 1, 2\}$. 
We put 
$(G, h)=(\yourysp_{\{0, 1, 2\}}, \yourydis_{\{0, 1, 2\}})$. 
Based on 
Theorems 
\ref{thm:18:product} 
and 
\ref{thm:18:hyper}, 
it is worth asking whether 
the 
isometry problems of   the 
cartesian product and 
the hyperspace of the random graph 
are true. 
Remark that 
if 
$p\in [1, \infty)$, 
then 
$h\times_{p}h$ 
does not take vales in 
$\{0, 1, 2\}$, 
and hence 
$(G\times G, h\times_{p}h)$ 
is not the random graph. 
In the case of 
$p=\infty$, 
similarly to 
Theorem 
\ref{thm:18:nonury}, 
the space 
$(G\times G, h\times_{p}h)$ 
is not 
isometric to $(G, h)$. 

\begin{ques}
Is there  a  function 
$F\colon [0, \infty)^{2}\to [0, \infty)$ 
of two variables such that 
$(G\times G, h\times_{F}h)$ 
 is 
isometric to 
$(G, h)$?
\end{ques}

\begin{ques}
Is 
$(\yoexpsp{G}, \yoexpdis{h})$ 
isometric to 
$(G, h)$?
\end{ques}

\begin{ac}
The author wishes to express his 
deepest gratitude 
to 
all  members of 
Photonics Control Technology Team (PCTT) in
 RIKEN, 
where
the majority  of the paper were written, 
for their invaluable   supports. 
Special thanks are extended to 
  the 
Principal Investigator of PCTT,
Satoshi Wada for 
the encouragement and support that 
transcended disciplinary boundaries.

This work was partially supported by JSPS 
KAKENHI Grant Number 
JP24KJ0182. 
\end{ac}


\bibliographystyle{myplaindoidoi}
\bibliography{../../../bibtex/UU.bib}


\end{document}